\documentclass{amsart}

\makeatletter
	
	\@addtoreset{equation}{section}

	\@addtoreset{figure}{section}
\makeatother

\usepackage{amsthm}
\usepackage{amsmath}
\usepackage{amsfonts}
\usepackage[pdftex]{graphicx}
\usepackage{subcaption}
\usepackage{url}
\usepackage{mathtools}
\usepackage{caption}
\title[]{On the Potential Function of the Colored Jones Polynomial with Arbitrary Colors}
\author[]{Shun Sawabe}
\address{Department of Pure and Applied Mathematics, School of Fundamental Science and Engineering, Waseda University, 3-4-1 Okubo, Shinjuku, Tokyo 169-8555, Japan}
\email{sa-shun.1729ttw@asagi.waseda.jp}
\date{}

\newtheorem{thm}{Theorem}[section]
\newtheorem{defi}[thm]{Definition}
\newtheorem{prop}[thm]{Proposition}

\newtheorem{conj}[thm]{Conjecture}

\theoremstyle{definition}
\newtheorem{rem}[thm]{Remark}
\newtheorem{ex}[thm]{Example}

\newcommand{\iu}{\sqrt{-1}}

\newcommand{\A}{\mathcal{A}}
\newcommand{\U}{\mathcal{U}}
\newcommand{\R}{\mathcal{R}}
\newcommand{\tb}{\overline{t}}

\DeclareMathOperator{\vol}{Vol}
\DeclareMathOperator{\li}{Li_2}

\DeclareMathOperator{\im}{Im}

\DeclareMathOperator{\arccosh}{arccosh}
\DeclareMathOperator{\lk}{lk}

\begin{document}
\begin{abstract}
	We consider the potential function of the colored Jones polynomial for a link with arbitrary colors
	and obtain the cone-manifold structure for the link complement. In addition, we
	establish a relationship between a saddle point equation and hyperbolicity of the link complement.
	This provides evidence for the Chen-Yang conjecture on the link complement.
\end{abstract}
\maketitle
\section{Introduction}
The volume conjecture is one of the most important problems in low-dimensional topology. 
Kashaev \cite{Ka} discovered that a certain limit of the Kashaev invariant of specific hyperbolic knots
such as the figure-eight knot is equal to the hyperbolic volume of their complements.
Murakami and Murakami \cite{MM} proved that the Kashaev invariant is a specialization of the colored Jones polynomial and conjectured
that a similar limit of the colored Jones polynomial for an arbitrary knot is equal to the simplicial volume of its complement.
In addition, Chen and Yang \cite{CY} considered the volume conjectures for 3-manifold invariants such as the Reshetikhin-Turaev invariant
and the Turaev-Viro invariant, and provided numerical evidence for them for specific 3-manifolds.
Detcherry, Kalfagianni, and Yang \cite{DKY} showed the relationship between the colored Jones polynomial for a link and the Turaev-Viro invariant of its
complement. By using this relation, they mathematically verified the Chen-Yang conjecture for complements of the figure-eight knot and Borromean rings.
In addition, Belletti, Detcherry, Kalfagianni, and Yang verified the Chen-Yang conjecture for fundamental shadow links in \cite{BDKY}.\par
Meanwhile, theoretical evidence of the original volume conjecture has been considered.
Kashaev and Tirkkonen \cite{KT} proved the volume conjecture for torus knots.
On the other hand, Yokota \cite{Yo} found a correspondence between quantum factorials in the Kashaev invariant and an ideal triangulation of a hyperbolic knot complement.
He showed that a saddle point equation for the potential function (see Section \ref{sec:pf} for the definition) of the invariant is equivalent to a hyperbolicity equation.
Also, the potential function of the colored Jones polynomial $J_N(K;e^{\frac{2 \pi \iu}{N}})$ for a hyperbolic knot $K$ is considered in \cite{Ch,Ch2,CM}.\par
In this study, we consider the potential function of
the colored Jones polynomial for a link $L$ with arbitrary colors.
We establish a relationship between a saddle point equation and a hyperbolicity equation of the link complement.
More precisely,
for a fixed diagram $D$ of the link $L$, we introduce a potential function
$\Phi_D(a_1,\ldots,a_n,w_1,\ldots,w_\nu)$ of the colored Jones polynomial $J_{\boldsymbol{a}(N)}(L;e^{\frac{2 \pi \iu}{N}})$
with parameters corresponding to the colors $\boldsymbol{a}(N)$ of link components.
We construct the potential function $\Phi_D$ by approximating the quantum $R$-matrix by continuous functions.
When we fix the parameters $ \boldsymbol{a}=(a_1,\ldots,a_n)$,
the saddle point $(\sigma_1(\boldsymbol{a}),\ldots,\sigma_\nu(\boldsymbol{a}))$
of $\Phi_D(\boldsymbol{a},-)$ gives a noncomplete hyperbolic structure to the link complement.
In fact, the manifold $M_{a_1,\ldots,a_n}$ with the hyperbolic structure is a cone-manifold.
Specifically, we prove the following statement:
\newtheorem*{mthm1}{\textrm \textbf Theorem~\ref{thm:m1}}
\begin{mthm1}
	The hyperbolic volume of the cone-manifold $M_{a_1,\ldots,a_n}$ is equal to the imaginary part of 
	\[
		\tilde{\Phi}_D = \Phi _D - \sum^{\nu}_{j=1}w_j \frac{\partial \Phi _D}{\partial w_j} \log w_j
	\]
	evaluated at $w_j = \sigma _j(\boldsymbol{a})$ for $j=1,\ldots,\nu $.
\end{mthm1}
Here, the function $ \Phi_D(\boldsymbol{a},\sigma_1(\boldsymbol{a}),\ldots,\sigma_\nu(\boldsymbol{a}))$
determines the Neumann-Zagier potential function \cite{NZ}.
Furthermore, we prove that the derivatives of the potential function
with respect to the new parameters correspond to the completeness of the hyperbolic structure of the link complement.
Note that similar arguments for the Kashaev invariant of the $5_2$ knot are indicated in \cite{Y2}.
As an application,
we prove the following theorem:
\newtheorem*{mthm2}{\textrm \textbf Theorem~\ref{thm:m2}}
\begin{mthm2}
	Let $D$ be a diagram of a hyperbolic link with $n$ components, and let $\boldsymbol{1}$ be $(1,\ldots,1) \in \mathbb{Z}^n$.
	The point $(\boldsymbol{1},\sigma_1(\boldsymbol{1}),\ldots,\sigma_\nu (\boldsymbol{1}))$ is a saddle point of the function
	$\Phi_D (a_1,\ldots,a_n,w_1,\ldots,w_\nu)$ and gives a complete hyperbolic structure to the link complement.
\end{mthm2}
The paper is organized as follows: In section \ref{sec:prelim}, we recall the facts on the colored Jones polynomial and the Turaev-Viro invariant.
In section \ref{sec:pf}, we give the potential function of the colored Jones polynomial.
In section \ref{sec:nchypstr}, we consider the case where the new parameters are fixed and prove Theorem \ref{thm:m1}.
In section \ref{sec:compness}, we regard the new parameters as variables and prove Theorem \ref{thm:m2}.
In section \ref{sec:wrt}, we briefly mention the Witten-Reshetikhin-Turaev invariant.
\par
\noindent \textit{Acknowledgments.} The author is grateful to Jun Murakami for his helpful comments.

\section{Preliminaries} \label{sec:prelim}
In this section, we review some facts on the invariants for a link and a $3$-manifold.
\subsection{The colored Jones polynomial and the Turaev-Viro invariant}
Let $L$ be an oriented $n$-component link, let $\boldsymbol{i}$ be a multiinteger, and let $t$ be an indeterminate. 
The colored Jones polynomial $J_{\boldsymbol{i}}(L;t)$ is defined skein-theoretically by using the Kauffman bracket,
which is a map $ \langle \cdot \rangle$ from the set of all unoriented diagrams of links to 
the ring of Laurent polynomials $\mathbb{Z}[A,A^{-1}]$ in an indeterminate $A$ given by the following axioms:
\begin{enumerate}
	\item For the trivial diagram $ \bigcirc $,
	\[
		\langle \bigcirc \rangle = 1.
	\]
	\item For an unoriented diagram $D$ with the trivial component added,
	\[
		\langle D \sqcup \bigcirc \rangle = (-A^2-A^{-2}) \langle D \rangle.
	\]
	\item For each crossing,
	\begin{center}
		\includegraphics[scale=1.0]{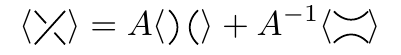}
	\end{center}
\end{enumerate}
Let $D_0$ be an unoriented diagram of the link $L$. The colored Jones polynomial $J_{\boldsymbol{i}}(L;t)$ for the link $L$
is a certain normalization of the Kauffman bracket of the
parallelized diagram of $D_0$ in which the Jones-Wenzl idempotent is inserted, where $t=A^{-4}$ \cite{DKY}.
\begin{rem}
	In this paper, we normalize the colored Jones polynomial so that the one for the trivial knot is equal to $1$.
\end{rem}
From the perspective of skein theory, we can define
the 3-manifold invariants such as the Reshetikhin-Turaev invariant or the Turaev-Viro invariant.
Detcherry, Kalfagianni, and Yang \cite{DKY} presented the relationship between the Turaev-Viro invariant for the link complement and the colored Jones polynomial.
\begin{thm}[Detcherry et al. \cite{DKY}] \label{thm:TVvsJ}
	Let $L \subset S^3$ be a link with $n$ components and $\tb=q^2$. Namely, $\tb=q^2=A^4$.
	\begin{enumerate}
		\item For an integer $r \geq 3$ and a primitive $4r$-th root of unity $A$,
		\[
			TV_r(S^3 \setminus L,q)= \eta _r^2 \sum _{1 \leq \boldsymbol{i}\leq r-1}|J'_{\boldsymbol{i}}(L;\tb)|^2.
		\]
		\item For an odd integer $r=2m+1 \geq 3$ and a primitive $2r$-th root of unity $A$,
		\[
			TV_r(S^3 \setminus L,q)= 2^{n-1}(\eta'_r)^2 \sum _{1 \leq \boldsymbol{i}\leq m}|J'_{\boldsymbol{i}}(L;\tb)|^2.
		\]
	\end{enumerate}
	Here, $ \eta _r$ and $ \eta'_r$ are
	\[
		\eta _r = \frac{A^2-A^{-2}}{\sqrt{-2r}},\quad \text{and} \quad \eta'_r = \frac{A^2-A^{-2}}{\sqrt{-r}}.
	\]
	In addition, for a multiinteger $\boldsymbol{i}=(i_1,\ldots,i_n)$,
	we let $ 1 \leq \boldsymbol{i}\leq m $ denotes that $1 \leq i_k \leq m $ for all integers $k = 1,\ldots,n$.
\end{thm}
\begin{rem}
	In \cite{DKY}, the normalization of the colored Jones polynomial and conventions on parameters are slightly different from the ones in this paper.
	Therefore, we use the notation $J'_{\boldsymbol{i}}(L;\tb)$ in Theorem  \ref{thm:TVvsJ}.
\end{rem}
These invariants are conjectured to relate to the geometry of the 3-manifold.
Murakami and Murakami \cite{MM} conjectured that a certain limit of the colored Jones polynomial for a knot
is equal to the volume of the complement of the knot.
\begin{conj}[Volume Conjecture \cite{MM}]
	For any knot $K$,
	\[
		2 \pi \lim _{N \to \infty} \frac{\log |J_N(K;t=e^{\frac{2 \pi \iu}{N}})|}{N} = v_3 ||K||,
	\]
	where $v_3$ is the volume of the ideal regular tetrahedron in the three-dimensional hyperbolic space
	and $|| \cdot ||$ is the simplicial volume for the complement of $K$.
\end{conj}
This conjecture was generalized to the one for $3$-manifold invariants.
\begin{conj}[Chen-Yang Conjecture \cite{CY}]
	For any $3$-manifold $M$ with a complete hyperbolic structure of the finite volume,
	\[
		2 \pi \lim _{r \to \infty}\frac{\log TV_r(M,q=e^{\frac{2 \pi \iu}{r}})}{r}= \vol(M),
	\]
	where $r$ runs over all odd integers, $TV(M)$ is a Turaev-Viro invariant of $M$ and $\vol(M)$ is a hyperbolic volume of $M$.
\end{conj}
Moreover, Detcherry, Kalfagianni, and Yang proved the following theorem by using Theorem \ref{thm:TVvsJ}:
\begin{thm}[Detcherry et al. \cite{DKY}] \label{thm:CYconjforL}
	Let $L$ be either the figure-eight knot or the Borromean rings, and let $M$ be the complement of $L$ in $S^3$. Then,
	\[
		2 \pi \lim _{r \to \infty}\frac{\log TV_r(M,q=e^{\frac{2 \pi \iu}{r}})}{r}=4 \pi \lim _{m \to \infty}\frac{\log |J'_{m}(L;\tb=e^{\frac{4 \pi \iu}{2m+1}})|}{2m+1} =\vol(M),
	\]
	where $r = 2m+1$ runs over all odd integers.
\end{thm}
\begin{rem}
	If $t$ is a root of unity, $\tb$ is the complex conjugate of $t$. Therefore,
	\[
		\lim _{m \to \infty}\frac{\log |J'_{m}(L;\tb=e^{\frac{4 \pi \iu}{2m+1}})|}{2m+1}=\lim _{m \to \infty}\frac{\log |J_{m}(L;t=e^{\frac{4 \pi \iu}{2m+1}})|}{2m+1}.
	\]
\end{rem}
Meanwhile, the evidence of the volume conjecture was established in \cite{Yo}.
What is important is that a saddle point equation of a potential function of the colored Jones polynomial for a knot coincides with 
a gluing condition of the ideal triangulation of the knot complement.
This and Theorem \ref{thm:TVvsJ} indicate that 
if we can establish a similar relationship between a hyperbolicity equation and a potential function of the colored Jones polynomial
with arbitrary colors, the relationship is evidence of the Chen-Yang conjecture for a link complement.
\subsection{The $R$-matrix of the colored Jones polynomial}
In this subsection, we give the $R$-matrix of the colored Jones polynomial by following \cite{KM}.
For an integer $r>1$, let $\A_r$ be the algebra generated by $X,\ Y,\ K,\ \overline{K}$ with the following relations:
\begin{align*}
	\overline{K}=K^{-1},\quad KX=sXK,\quad &KY=s^{-1}YK,\quad XY-YX=\frac{K^2-\overline{K}^2}{s-s^{-1}},\\
	&X^r=Y^r=0,\quad K^{4r}=1,
\end{align*}
where $s=e^{\frac{\pi \iu}{r}}$.
Namely, $\A_r$ is $\U_q(\mathrm{sl}_2)$ with the last $3$ relations. The universal $R$-matrix $\R \in \A_r \otimes \A_r$ is given by
\[
	\R=\frac{1}{4r}\sum _{\substack{0 \leq k < r \\ 0 \leq a,b <4r}}\frac{(s-s^{-1})^k}{[k]_s !}s^{-\frac{1}{2}(ab+(b-a)k+k)}X^kK^a \otimes Y^k K^b.
\]
Here, we put
\[
	[k]_s=\frac{s^k-s^{-k}}{s-s^{-1}},\quad [k]_s !=[k]_s \cdots [1]_s, \quad [0]_s!=1.
\]
Let $N$ be a positive integer and $m$ be the half-integer satisfying $N=2m+1$.
We define the action of $\A_r$ on an $N$-dimensional complex vector space $V$ with a basis $\{e_{-m},e_{-m+1},\ldots,e_m\}$ by
\begin{align*}
	Xe_i &= [m-i+1]_s e_{i-1}, \\
	Ye_i &= [m+i+1]_s e_{i+1}, \\
	Ke_i &= s^{-i}e_i.
\end{align*}
Here, $e_i$ in this paper corresponds to $e_{-i}$ in \cite{KM}.
Let $V'$ be an $(N'=2m'+1)$-dimensional complex vector space with basis $\{e'_{-m'},\ldots,e'_{m'}\}$.
Then, the quantum $R$-matrix $R_{VV'}:V \otimes V' \to V' \otimes V$ is given by
\begin{align*}
	R_{VV'}(e_i \otimes e'_j)=\sum _{k=0}^{\min \{m+i,\ m'-j \}} &\frac{\{m-i+k\}_s!\{m'+j+k\}_s!}{\{k\}_s!\{m-i\}_s!\{m'+j\}_s!} \\
	& \times s^{2ij+k(i-j)-k(k+1)/2}e'_{j+k} \otimes e_{i-k},
\end{align*}
where
\[
	\{k\}_s = s^{k}-s^{-k}, \quad \{k\}_s!=\{k\}_s \cdots \{1\}_s,\quad \{0\}_s! = 1.
\]
Also, its inverse is
\begin{align*}
	R_{VV'}^{-1}(e'_i \otimes e_j) = \sum _{k=0}^{\min \{m-i,\ m'+j \}} &(-1)^k \frac{\{m-j+k\}_s!\{m'+i+k\}_s!}{\{k\}_s!\{m-j\}_s!\{m'+i\}_s!} \\
	& \times s^{-2ij+k(i-j)/2+k(k+1)/2}e_{j-k} \otimes e'_{i+k}.
\end{align*}
These matrices and the isomorphism $\mu: V \to V$, where
\[
	\mu(e_i)=s^{-2i}e_i,\quad(i=-m,\ldots,m)
\]
defines a link invariant $\tilde{J}$.
If $V=V'$ and $\dim V = 2$, then
\[
	R_{VV}=\left(\begin{array}{cccc}
		s^{\frac{1}{2}}&	0&	0&	0\\
		0&	0&	s^{-\frac{1}{2}}&	0\\
		0&	s^{-\frac{1}{2}}&	s^{\frac{1}{2}}-s^{-\frac{3}{2}}&	0 \\
		0&	0&	0&	s^{\frac{1}{2}}
	\end{array} \right)
\]
and satisfies
\[
	s^{\frac{1}{2}}R_{VV} - s^{-\frac{1}{2}}R_{VV}^{-1}= (s-s^{-1})I_4,
\]
where $I_4$ is the $4 \times 4$ identity matrix.
Considering the writhes, this implies
\begin{equation} \label{eq:Jske}
	s^2 \tilde{J}(L_+)-s^{-2} \tilde{J}(L_-)=(s-s^{-1})\tilde{J}(L_0),
\end{equation}
where $L_+,\ L_-$, and $L_0$ are the links in Figure \ref{fig:LLL}.
	\begin{figure}[htb]
		\centering
		\includegraphics[scale=0.7]{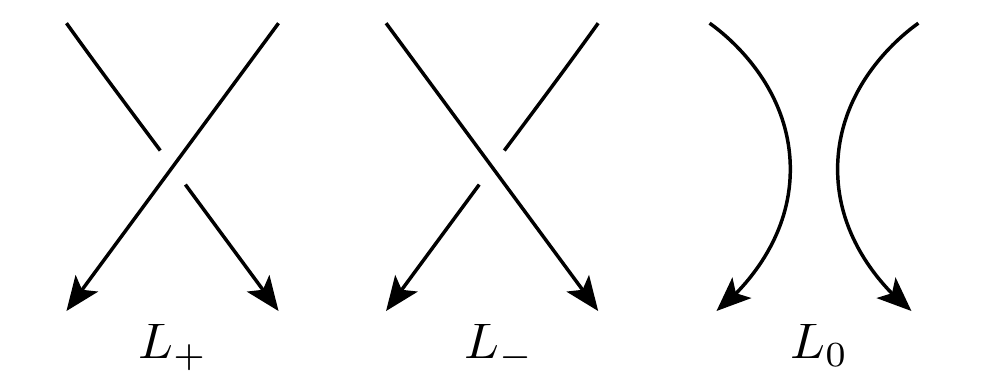}
		\caption{The links that are identical except for these regions.}
		\label{fig:LLL}
	\end{figure}
Under the substitution $s=-t^{-\frac{1}{2}}$, the relation \eqref{eq:Jske} coincides with the skein relation of the Jones polynomial.
Therefore, under this substitution the $R$-matrix of the colored Jones polynomial $J_{\boldsymbol{i}}(L;t)$
for $L$ with colors $\boldsymbol{i}=(i_1,\ldots,i_n) \in \mathbb{Z}^n_{>0}$, where $i_j$, with $j=1,\ldots,n$, is the dimension of the assigned representation is
\begin{align} \label{eq:Rmat}
\begin{split}
	R_{VV'}(e_i \otimes e'_j)=\sum _{k=0}^{\min \{m+i,\ m'-j \}} &(-1)^{k+k(m+m')+2ij}\frac{\{m-i+k\}!\{m'+j+k\}!}{\{k\}!\{m-i\}!\{m'+j\}!} \\
	& \times t^{-ij-\frac{k(i-j)}{2}+\frac{k(k+1)}{4}}e'_{j+k} \otimes e_{i-k},
\end{split}
\end{align}
and its inverse is
	\begin{align*}
		R_{VV'}^{-1}(e'_i \otimes e_j) = \sum _{k=0}^{\min \{m-i,\ m'+j \}} &(-1)^{-k(m+m')-2ij} \frac{\{m-j+k\}!\{m'+i+k\}!}{\{k\}!\{m-j\}!\{m'+i\}!} \\
		& \times t^{ij-\frac{k(i-j)}{2}-\frac{k(k+1)}{4}}e_{j-k} \otimes e'_{i+k},
	\end{align*}
where
\[
	\{k\} = t^{\frac{k}{2}}-t^{-\frac{k}{2}}, \quad \{k\}!=\{k\}\{k-1\}\cdots \{1\},\quad \{0\}! = 1.
\]

\section{Potential function} \label{sec:pf}
Let $L=L_1 \cup \cdots \cup L_n$ be an oriented $n$-component link. We deform $L$ so that $L$ is a closure of a braid.
Let $D$ be its oriented diagram, and $\xi_N=e^{\frac{2\pi \iu}{N}}$ be the primitive $N$-th root of unity.
For each link component $L_i$, with $i=1,\ldots,n$, we assign its color $a_i(N) \in \mathbb{Z}_{>0}$. We put $\boldsymbol{a}(N)=(a_1(N),\ldots,a_n(N))$.
In this section, we determine a potential function of the colored Jones polynomial
$J_{\boldsymbol{a}(N)}(L;\xi_N^p)$ for $L$, where $p$ is a nonzero integer.
See \cite{Ch} for details.
\begin{defi}
	Suppose that the asymptotic behavior of a certain quantity $Q_N$ for a sufficiently large $N$ is 
	\[
		Q_N \sim \int \cdots \int _\Omega P_N e^{\frac{N}{2 \pi \iu} \Phi(z_1,\ldots,z_\nu)}dz_1 \cdots dz_\nu,
	\]
	where $P_N$ grows at most polynomially and $ \Omega $ is a region in $\mathbb{C}^\nu$. We call this function $\Phi(z_1,\ldots,z_\nu)$ a potential function of $Q_N$.
\end{defi}
We can easily verify that
\begin{equation} \label{eq:qfbr}
	\{k\}! = (-1)^k t^{-\frac{k(k+1)}{4}}(t)_k,
\end{equation}
where
\[
	(t)_k =(1-t)(1-t^2) \cdots (1-t^k).
\]
Thus, we approximate $(\xi_N^p)_k$ by continuous functions.
\begin{prop} \label{prop:approx}
	For a sufficiently large integer $N$,
	\[
		\log(\xi^p_N)_k = \frac{N}{2p \pi \iu}\left( -\li(\xi^{pk}_N)+\frac{\pi^2}{6}+o(1)\right),
	\]
	where $\li$ is a dilogarithm function
	\[
		\li(z) = -\int^z_0 \frac{\log(1-x)}{x}dx.
	\]
\end{prop}
\begin{rem}
	The dilogarithm function satisfies
	\[
		\li(z) = \sum^{\infty}_{k=1}\frac{z^2}{k^2}
	\]
	for $|z|<1$, and
	\[
		\li(1) = \sum^{\infty}_{k=1}\frac{1}{k^2} = \frac{\pi^2}{6}.
	\]
\end{rem}
\begin{proof}
	By the direct calculation, we have
	\begin{align*}
			\log(\xi^p_N)_k &= \sum^k_{j=1}\log(1-e^{\frac{2p \pi j \iu }{N}}) \\
			&= N \left( \int^{\frac{k}{N}}_0 \log(1-e^{2p \pi \iu \theta})d \theta+o(1) \right)\\
			&=\frac{N}{2p \pi \iu}\left( \int^{\xi^{pk}_N}_1 \frac{\log(1-x)}{x}dx +o(1) \right)\\
			&= \frac{N}{2p \pi \iu}\left( -\li(\xi^{pk}_N)+\frac{\pi^2}{6}+o(1)\right).
	\end{align*}
\end{proof}
First, we consider the case where the strings at a crossing are in the different components.
Let $\{a(N)\}_{N=1,2,\ldots}$ and $\{b(N)\}_{N=1,2,\ldots}$ be sequences of natural numbers.
We can approximate the $R$-matrix by Proposition \ref{prop:approx}.
For a positive crossing of the link diagram, the $R$-matrix $R_{VV'}$ of \eqref{eq:Rmat} is labeled.
For convenience, we recall the summand of the $R$-matrix:
\[
	(-1)^{k+k(m_N+m'_N)+2ij}t^{-ij-\frac{k(i-j)}{2}+\frac{k(k+1)}{4}}\frac{\{m_N-i+k\}!\{m'_N+j+k\}!}{\{k\}!\{m_N-i\}!\{m'_N+j\}!}.
\]
Here, $m_N$ and $m'_N$ are the half-integers satisfying $a(N) = 2m_N +1$ and $b(N)=2m'_N +1$.
If we assume that $a(N)$ and $b(N)$ are odd numbers, indices $i$ and $j$ are integers.
Moreover, by adding $2$ to $a(N)$ or $b(N)$ if necessary, we can assume that $m_N+m'_N$ is an even integer
without changing the values of the limit $a(N)/N$ and $b(N)/N$. Therefore, under these assumptions the summand is
\[
	(-1)^{k}t^{-ij-\frac{k(i-j)}{2}+\frac{k(k+1)}{4}}\frac{\{m_N-i+k\}!\{m'_N+j+k\}!}{\{k\}!\{m_N-i\}!\{m'_N+j\}!}.
\]
From the equality \eqref{eq:qfbr}, we have
\[
	t^{-ij-\frac{m_N+m'_N}{2} k}\frac{(t)_{m_N -i+k}(t)_{m'_N +j+k}}{(t)_{k}(t)_{m_N-i}(t)_{m'_N+j}}. 
\]
Under substitution $x=\xi_N^i$, $y=\xi_N^j$ and $z=\xi_N^k$, the potential function for a positive crossing is
	\begin{align*}
		\frac{1}{p} &\left\{-\pi \iu p \frac{a+b}{2} \log(z^p)- \log(x^p)\log(y^p) - \frac{\pi^2}{6} \right.\\
		&-\left. \li \left(e^p_a \frac{z^p}{x^p} \right) -\li (e^p_b y^pz^p)  +\li \left(\frac{e^p_a}{x^p} \right) +\li(e^p_b y^p) +\li(z^p) \right\},
	\end{align*}
where $a(N)/N \to a$, $b(N)/b \to b$ and $e_a = e^{\pi \iu a}$.
Note that the indices of the summand are labeled to the edges of the link diagram.
We change these indices to the ones corresponding to regions of the link diagram as shown in Figure \ref{fig:regionedge}.
	\begin{figure}[htb]
		\centering
		\includegraphics[scale=1.0]{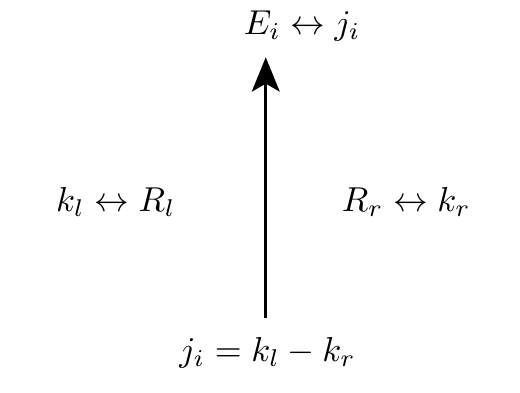}
		\caption{Indices corresponding to an edge $E_i$ and regions $R_l,\ R_r$.}
		\label{fig:regionedge}
	\end{figure}
If $k_{j_1},\ldots,k_{j_4}$ are indices around a crossing as shown in Figure \ref{fig:regionindices}, we have
\[
	i=k_{j_2}-k_{j_1},\quad j=k_{j_3}-k_{j_2},\quad j+k=k_{j_4}-k_{j_1},\quad i-k=k_{j_3}-k_{j_4}.
\]
From the above equations, we have $k=k_{j_2}+k_{j_4}-k_{j_1}-k_{j_3}$. Therefore by putting $w_{j_i}=\xi_N^{k_{j_i}}$ and substituting
\[
x=\frac{w_{j_2}}{w_{j_1}},\quad y=\frac{w_{j_3}}{w_{j_2}},\quad z=\frac{w_{j_2}w_{j_4}}{w_{j_1}w_{j_3}},
\]
the potential function for a positive crossing $c$ is
	\begin{align*}
		\Phi_{c,p}^+&=\frac{1}{p} \left\{\pi \iu p^2 \frac{a+b}{2} \log \frac{w_{j_1}w_{j_3}}{w_{j_2}w_{j_4}}- p^2 \log \frac{w_{j_2}}{w_{j_1}} \log \frac {w_{j_3}}{w_{j_2}} \right.\\
		&-\li \left(e^p_a \frac{w_{j_4}^p}{w_{j_3}^p}\right)-\li \left(e^p_b \frac{w_{j_4}^p}{w_{j_1}^p}\right)+\li \left(\frac{w_{j_2}^p w_{j_4}^p}{w_{j_1}^p w_{j_3}^p} \right)\\
		&+\li \left(e^p_a \frac{w_{j_1}^p}{w_{j_2}^p}\right)+\left.\li \left(e^p_b \frac{w_{j_3}^p}{w_{j_2}^p} \right)-\frac{\pi^2}{6}\right\}.
	\end{align*}
	\begin{figure}[htb]
		\centering
		\includegraphics[scale=0.7]{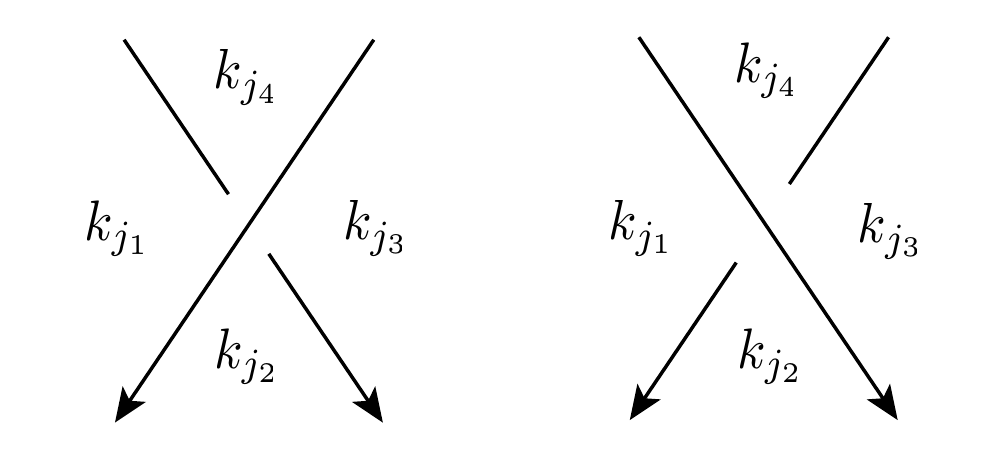}
		\caption{Indices corresponding to regions around a crossing.}
		\label{fig:regionindices}
	\end{figure}
If the strings at a crossing are in the same component, we have to consider the modification on the Reidemeister move I.
The Reidemeister  move I on the component with a color $a(N)$ leads to the multiplication by $s^{2m_N^2+2m_N}=(-1)^{2m_N^2+2m_N}t^{-m_N^2-m_N}$.
Therefore, we have to multiply $(-1)^{-2m_N^2-2m_N}t^{m_N^2+m_N}$ to cancel it.
Under the assumption that $a(N)$ is an odd integer, this corresponds to the addition of the function $(\pi \iu pa)^2/p$.
Therefore, the potential function is
	\begin{align*}
		&\Phi_{c,p}^+ =\frac{1}{p} \left\{(\pi \iu pa)^2 + \pi \iu p^2 a \log \frac{w_{j_1}w_{j_3}}{w_{j_2}w_{j_4}}- p^2 \log \frac{w_{j_2}}{w_{j_1}} \log \frac {w_{j_3}}{w_{j_2}} -\frac{\pi^2}{6}\right.\\
		&-\li \left(e^p_a \frac{w_{j_4}^p}{w_{j_3}^p}\right)-\li \left(e^p_a \frac{w_{j_4}^p}{w_{j_1}^p}\right)+\li \left(\frac{w_{j_2}^p w_{j_4}^p}{w_{j_1}^p w_{j_3}^p} \right)+\li \left(e^p_a \frac{w_{j_1}^p}{w_{j_2}^p}\right)+\left.\li \left(e^p_a \frac{w_{j_3}^p}{w_{j_2}^p} \right)\right\}.
	\end{align*}
Similarly, we obtain
	\begin{align*}
		\Phi_{c,p}^-&=\frac{1}{p} \left\{-\pi \iu p^2 \frac{a+b}{2} \log \frac{w_{j_1}w_{j_3}}{w_{j_2}w_{j_4}}+ p^2 \log \frac{w_{j_3}}{w_{j_4}} \log \frac{w_{j_4}}{w_{j_1}} \right.\\
		&-\li \left(e^p_a \frac{w_{j_1}^p}{w_{j_4}^p} \right)-\li \left(e^p_b\frac{w_{j_3}^p}{w_{j_4}^p} \right)-\li \left(\frac{w_{j_2}^p w_{j_4}^p}{w_{j_1}^p w_{j_3}^p} \right)\\
		&+\left.\li \left(e^p_a \frac{w_{j_2}^p}{w_{j_3}^p}\right)+\li \left(e^p_b \frac{w_{j_2}^p}{w_{j_1}^p}\right)+\frac{\pi^2}{6}\right\}
	\end{align*}
for a negative crossing $c$ between different components, and
	\begin{align*}
		&\Phi_{c,p}^-=\frac{1}{p} \left\{-(\pi \iu pa)^2 -\pi \iu p^2 a \log \frac{w_{j_1}w_{j_3}}{w_{j_2}w_{j_4}}+ p^2 \log \frac{w_{j_3}}{w_{j_4}} \log \frac{w_{j_4}}{w_{j_1}}+\frac{\pi^2}{6} \right.\\
		&-\li \left(e^p_a \frac{w_{j_1}^p}{w_{j_4}^p} \right)-\li \left(e^p_a\frac{w_{j_3}^p}{w_{j_4}^p} \right)-\li \left(\frac{w_{j_2}^p w_{j_4}^p}{w_{j_1}^p w_{j_3}^p} \right)+\left.\li \left(e^p_a \frac{w_{j_2}^p}{w_{j_3}^p}\right)+\li \left(e^p_a \frac{w_{j_2}^p}{w_{j_1}^p}\right)\right\}
	\end{align*}
for a negative crossing $c$ between the same component.
The potential function $\Phi_{D,p}$ of $J_{\boldsymbol{a}(N)}(L,\xi^p_N)$ is a summation of these
potential functions with respect to all crossings of $D$. That is,
\[
	\Phi _{D,p}(\boldsymbol{a},w_1,\ldots,w_\nu)= \sum _{c\text{ is a crossing}}\Phi^{\mathrm{sgn}(c)}_{c,p},
\]
where
\[
	\boldsymbol{a} = (a_1 ,\ldots,a_n ),\quad a_i=\lim_{N \to \infty} \frac{a_i(N)}{N}
\]
and $\mathrm{sgn}(c)$ is a signature of a crossing $c$.
This potential function essentially coincides with Yoon's generalized potential function \cite{Yon}.
We can easily verify the following property by the definition of $ \Phi_{D,p}$:
\begin{prop}
	$\Phi _{D,p}(\boldsymbol{a},w_1,\ldots,w_\nu)$ satisfies
	\[
		\Phi_{D,p}(\boldsymbol{a},w_1,\ldots,w_\nu) = \frac{1}{p}\Phi_{D,1}(p \boldsymbol{a},w_1^p,\ldots,w_\nu^p).
	\]
\end{prop}
Therefore, We mainly consider the case where $p=1$ and write $ \Phi _D = \Phi_{D,1}$.

\section{A Non-complete hyperbolic structure} \label{sec:nchypstr}

In this section, we provide geometric meanings of the potential function.
In the rest of this paper, we assume that $L$ is a hyperbolic link with $n$ components.
In this section, we also assume that $a_i \in [1-\varepsilon,1]$ for all $i=1,\ldots,n$,
where $ \varepsilon $ is a sufficiently small positive real number.
First, we consider derivatives of the potential functions with respect to the parameters corresponding to the regions of the link diagram \cite{CM}.
For a positive crossing $c$ between different components, we have
	\begin{align} \label{eq:pglu}
	\begin{split}
		w_{j_1}\frac{\partial \Phi^+_c}{\partial w_{j_1}}&= \pi \iu \frac{a-b}{2} + \log \left(1-e_a\frac{w_{j_1}}{w_{j_2}}\right)^{-1}\left(1-e_b^{-1}\frac{w_{j_1}}{w_{j_4}}\right)^{-1}\left( 1-\frac{w_{j_1}w_{j_3}}{w_{j_2}w_{j_4}}\right),\\
		w_{j_2}\frac{\partial \Phi^+_c}{\partial w_{j_2}}&= \pi \iu \frac{a+b}{2} + \log \left(1-e_a^{-1}\frac{w_{j_2}}{w_{j_1}}\right)\left(1-e_b^{-1}\frac{w_{j_2}}{w_{j_3}}\right)\left( 1-\frac{w_{j_2}w_{j_4}}{w_{j_1}w_{j_3}}\right)^{-1},\\
		w_{j_3}\frac{\partial \Phi^+_c}{\partial w_{j_3}}&= \pi \iu \frac{-a+b}{2} + \log \left(1-e_a^{-1}\frac{w_{j_3}}{w_{j_4}}\right)^{-1}\left(1-e_b \frac{w_{j_3}}{w_{j_2}}\right)^{-1}\left( 1-\frac{w_{j_1}w_{j_3}}{w_{j_2}w_{j_4}}\right),\\
		w_{j_4}\frac{\partial \Phi^+_c}{\partial w_{j_4}}&= -\pi \iu \frac{a+b}{2} + \log \left(1-e_a \frac{w_{j_4}}{w_{j_3}}\right)\left(1-e_b\frac{w_{j_4}}{w_{j_1}}\right)\left( 1-\frac{w_{j_2}w_{j_4}}{w_{j_1}w_{j_3}}\right)^{-1}.
	\end{split}
	\end{align}
Similarly, for a negative crossing $c$ between different components, we have
	\begin{align} \label{eq:nglu}
	\begin{split}
		w_{j_1}\frac{\partial \Phi^-_c}{\partial w_{j_1}}&= \pi \iu \frac{-a+b}{2} + \log \left(1-e_a \frac{w_{j_1}}{w_{j_4}}\right)\left(1-e_b^{-1} \frac{w_{j_1}}{w_{j_2}}\right)\left( 1-\frac{w_{j_1}w_{j_3}}{w_{j_2}w_{j_4}}\right)^{-1},\\
		w_{j_2}\frac{\partial \Phi^-_c}{\partial w_{j_2}}&= \pi \iu \frac{a+b}{2} + \log \left(1-e_a \frac{w_{j_2}}{w_{j_3}}\right)^{-1}\left(1-e_b \frac{w_{j_2}}{w_{j_1}}\right)^{-1}\left( 1-\frac{w_{j_2}w_{j_4}}{w_{j_1}w_{j_3}}\right),\\
		w_{j_3}\frac{\partial \Phi^-_c}{\partial w_{j_3}}&= \pi \iu \frac{a-b}{2} + \log \left(1-e_a^{-1} \frac{w_{j_3}}{w_{j_2}}\right)\left(1-e_b \frac{w_{j_3}}{w_{j_4}}\right)\left( 1-\frac{w_{j_1}w_{j_3}}{w_{j_2}w_{j_4}}\right)^{-1},\\
		w_{j_4}\frac{\partial \Phi^-_c}{\partial w_{j_4}}&= -\pi \iu \frac{a+b}{2} + \log \left(1-e_a^{-1}\frac{w_{j_4}}{w_{j_1}}\right)^{-1}\left(1-e_b^{-1}\frac{w_{j_4}}{w_{j_3}}\right)^{-1}\left( 1-\frac{w_{j_2}w_{j_4}}{w_{j_1}w_{j_3}}\right).
	\end{split}
	\end{align}
If a crossing is between the same component, the derivatives are \eqref{eq:pglu} and \eqref{eq:nglu} with $a=b$.
These correspond to Thurston's triangulation \cite{Dt} of the link complement (see Figures \ref{fig:pcrossoct}, \ref{fig:ncrossoct}).
\begin{figure}[tb]
	\begin{tabular}{cc}
		\begin{minipage}[t]{0.45\hsize}
		\captionsetup{width=5.5cm}
		\centering
		\includegraphics[scale=1.0]{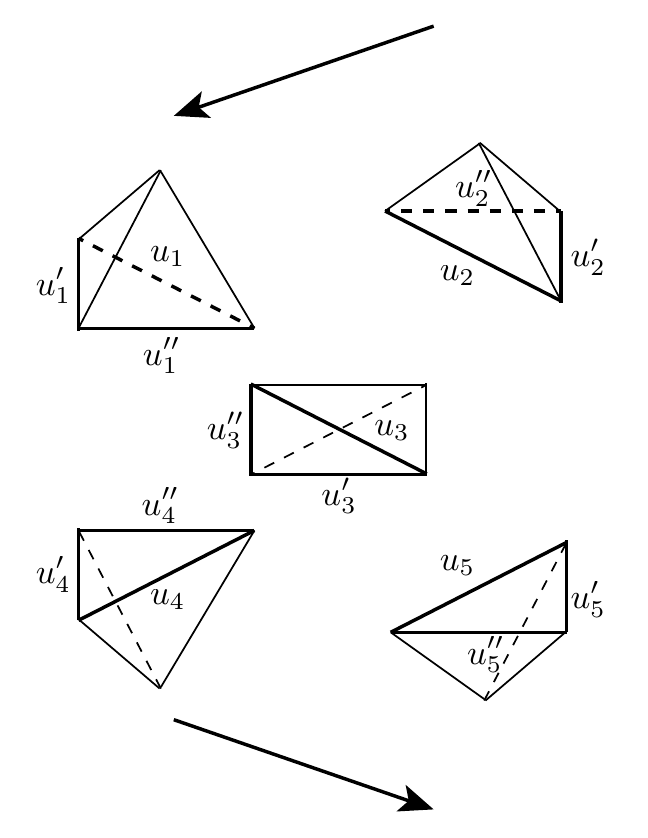}
		\caption{Ideal tetrahedra on a positive crossing.}
		\label{fig:pcrossoct}
		\end{minipage}
	
		\begin{minipage}[t]{0.45\hsize}
		\captionsetup{width=5.5cm}
		\centering
		\includegraphics[scale=1.0]{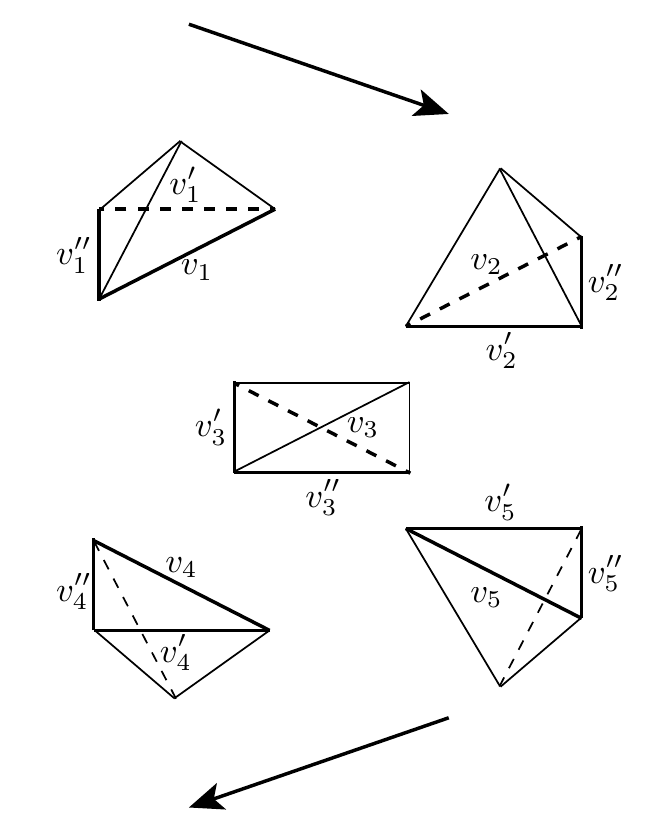}
		\caption{Ideal tetrahedra on a negative crossing.}
		\label{fig:ncrossoct}
		\end{minipage}
	\end{tabular}
\end{figure}
Here, we put
\[
	u_1=e_a \frac{w_{j_1}}{w_{j_2}},\quad u_2=e_a^{-1}\frac{w_{j_3}}{w_{j_4}},\quad u_3=\frac{w_{j_2}w_{j_4}}{w_{j_1}w_{j_3}},\quad u_4=e_b^{-1}\frac{w_{j_1}}{w_{j_4}},\quad u_5=e_b \frac{w_{j_3}}{w_{j_2}},
\]
\[
	v_1=e_a^{-1}\frac{w_{j_4}}{w_{j_1}},\quad v_2=e_a \frac{w_{j_2}}{w_{j_3}},\quad v_3=\frac{w_{j_1}w_{j_3}}{w_{j_2}w_{j_4}},\quad v_4=e_b \frac{w_{j_2}}{w_{j_1}},\quad v_5=e_b^{-1}\frac{w_{j_4}}{w_{j_3}} 
\]
in Figures \ref{fig:pcrossoct} and \ref{fig:ncrossoct}.
Furthermore, for a comlex number $z$, denote 
\[
	z' = \frac{1}{1-z}\quad \text{and} \quad z''=1-\frac{1}{z}.
\]
Note that if there exists a nonalternating part, the ideal tetrahedron abuts the one with the inverse complex number labeled.
Thus we can ignore the contribution of such a part.
Let $G_i$ be a product of the parameters of ideal tetrahedra around the region $R_i$ corresponding to the parameter $w_i$.
Then, we have
\[
	w_i \frac{\partial \Phi _D}{\partial w_i} = \frac{\pi \iu}{2}r(a_1,\ldots,a_n) + \log G_i,
\]
where $\pi \iu r(a_1,\ldots,a_n)/2$ is the summation of first terms of $w_{i}\partial\Phi_c^\pm/\partial w_{i}$
with $c$ running over all crossings around $R_i$. 
However, this is equal to $0$ because the contribution of each parameter $a$ to $r(a_1,\ldots,a_n)$ is canceled
as in Figures \ref{fig:crosspm} and \ref{fig:colcontrib}.
\begin{figure}[htb]
	\centering
		\begin{minipage}{0.45\hsize}
		\captionsetup{width=5.5cm}
		\centering
		\includegraphics[height=3.0cm]{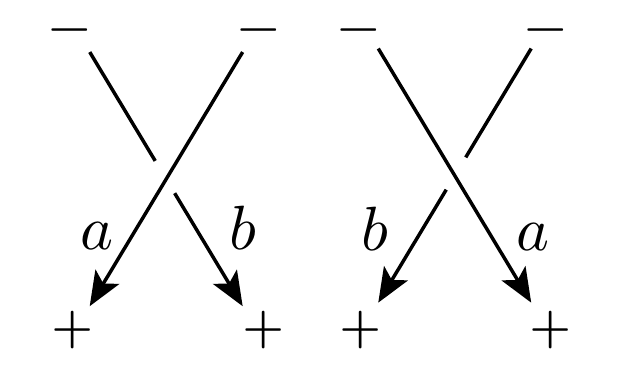}
		\caption{Signatures of parameters corresponding to edges. Note that the pattern of signatures is independent of the signature of a crossing.}
		\label{fig:crosspm}
		\end{minipage}
		\begin{minipage}{0.45\hsize}
		\captionsetup{width=5.5cm}
		\centering
		\includegraphics[height=3.0cm]{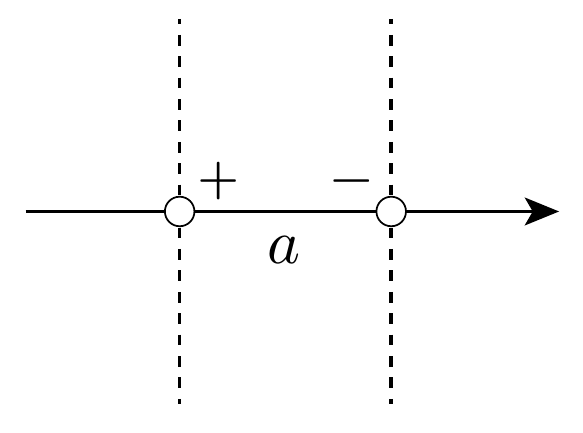}
		\caption{Contributions of each parameter. White circles represent either positive or negative crossings.}
		\label{fig:colcontrib}
		\end{minipage}
\end{figure}
Therefore, the equations
\begin{equation} \label{eq:lsaddle}
	\exp \left( w_i \frac{\partial \Phi _D}{\partial w_i} \right)= 1, \quad i=1,2,\ldots,\nu
\end{equation}
coincide with the gluing condition of the ideal tetrahedra.
Hence, we can obtain a hyperbolic structure from a saddle point $(\sigma_1(\boldsymbol{a}),\ldots,\sigma_\nu(\boldsymbol{a}))$ of $\Phi_D(\boldsymbol{a},-)$,
where $\boldsymbol{a}=\left(a_1,\ldots,a_n\right)$.
In addition, this hyperbolic structure is not complete in general because the dilation component of the meridian of the link component
with the color $a$ is equal to $e_a^{-2}$ (see Figure \ref{fig:cdmerid}).
	\begin{figure}[htb]
		\centering
		\includegraphics[scale=0.7]{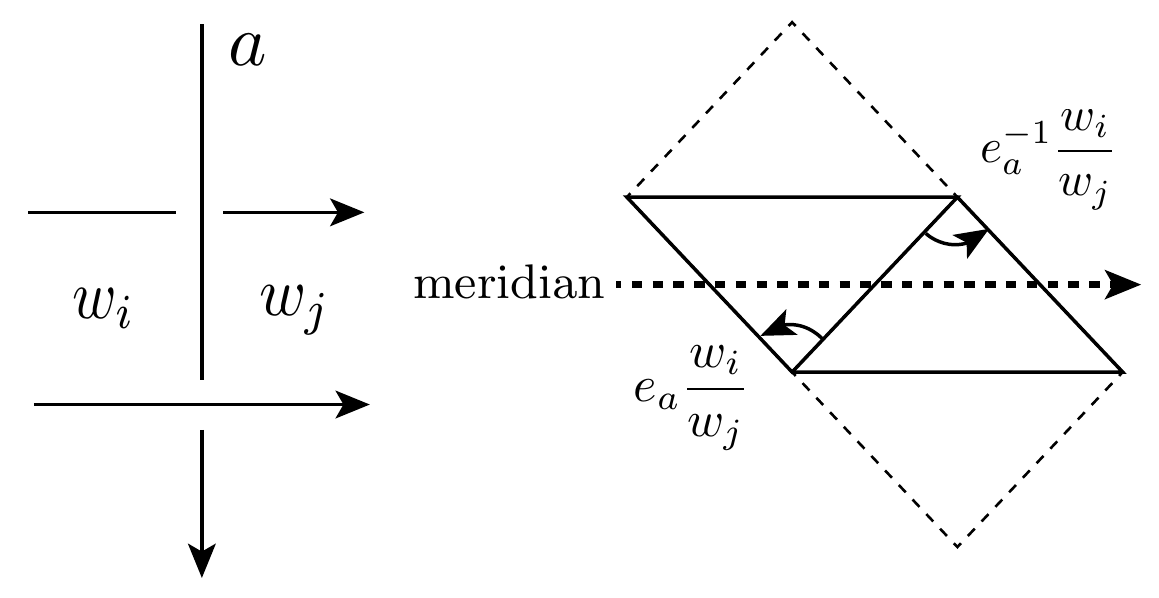}
		\caption{The dilation component of the meridian of the link component with the parameter $a$.}
		\label{fig:cdmerid}
	\end{figure}
Note that $\boldsymbol{a}=(1,\ldots,1)$ is the case of the original volume conjecture.
So we suppose that $(\sigma_1(\boldsymbol{a}),\ldots,\sigma_\nu(\boldsymbol{a}))$ gives $S^3 \setminus L$ the hyperbolic structure with the finite volume $\vol(S^3 \setminus L)$
when $\boldsymbol{a}=(1,\ldots,1)$ \cite{Ch2}.
Let $M_{a_1,\ldots,a_n}$ be a manifold with the hyperbolic structure given by $(\sigma_1(\boldsymbol{a}),\ldots,\sigma_\nu(\boldsymbol{a}))$.
We will determine the detail of this non-complete hyperbolic manifold $M_{a_1,\ldots,a_n}$.
Let $a$ be a real number slightly less than $1$. Note that the action derived from each meridian does not change a length because $|e_a^{-2}|=1$.
Therefore, the action derived from each longitude changes a length, since otherwise, both meridians and longitudes do not change a length 
and this results in the complete hyperbolic structure \cite{BP}.
Therefore, the developing image in the upper half-space $\mathbb{H}^3$ of the link complement around the edge corresponding to parameter $a$
should be as shown in Figure \ref{fig:develop}.
	\begin{figure}[htb]
		\centering
		\includegraphics[scale=0.7]{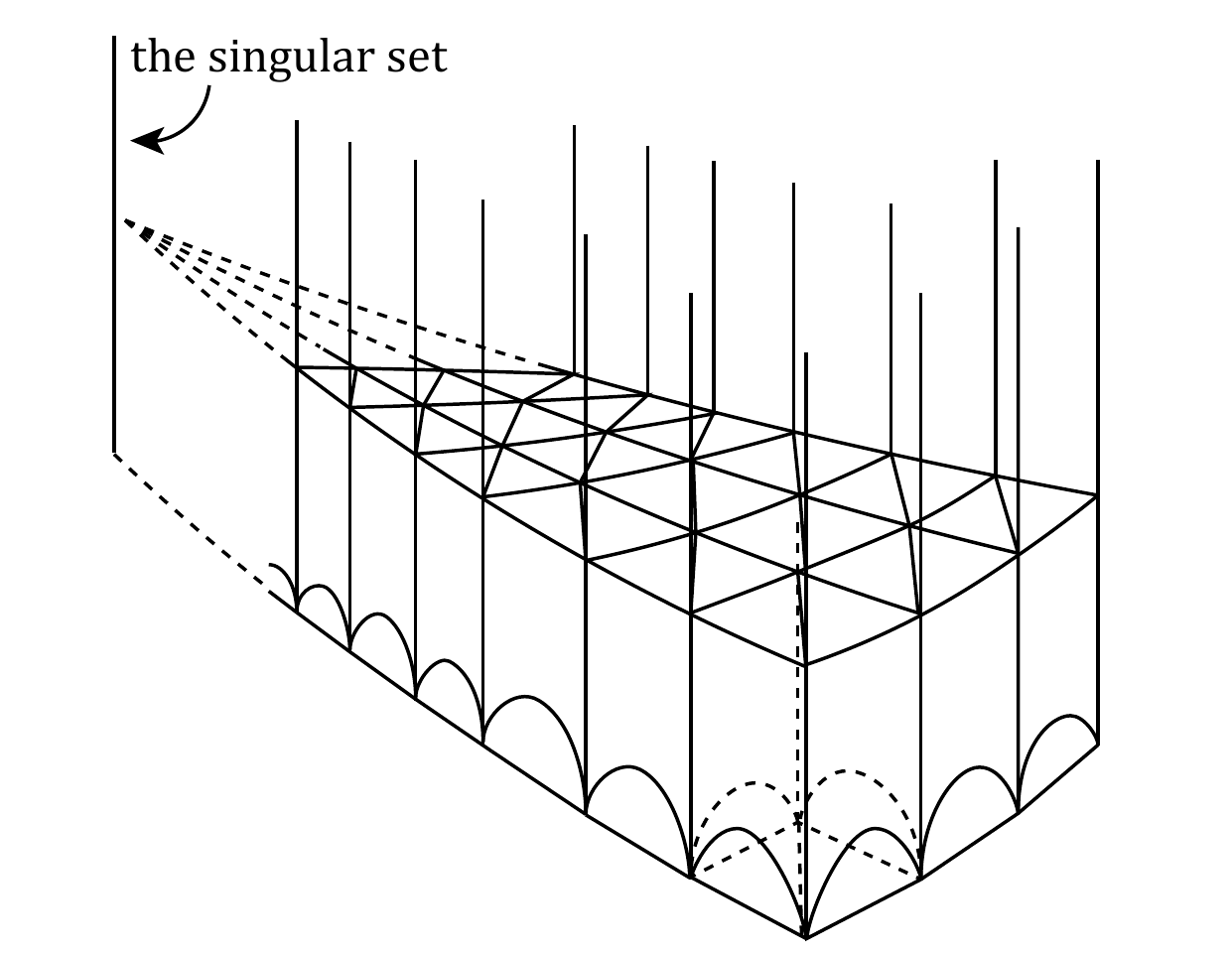}
		\caption{The developing image of the link complement with the non-complete hyperbolic structure.}
		\label{fig:develop}
	\end{figure}
If we glue faces by the action of meridians in Figure \ref{fig:develop}, each face is glued with the face rotated $2\pi(1-a)$ around the singular set.
Therefore, $M_{a_1,\ldots,a_n}$ is a cone-manifold of $L$ with cone-angle $2\pi(1-a_i)$ around the component corresponding to $a_i$.
Specifically, we can prove the following proposition:
\begin{thm} \label{thm:m1}
	The hyperbolic volume of the cone-manifold $M_{a_1,\ldots,a_n}$ is equal to the imaginary part\footnote{In \cite{Mh}, this value is called the optimistic limit.} of a function
	\[
		\tilde{\Phi}_D = \Phi _D - \sum^{\nu}_{j=1}w_j \frac{\partial \Phi _D}{\partial w_j} \log w_j
	\]
	evaluated at $w_j = \sigma _j(\boldsymbol{a})$, with $j=1,\ldots,\nu $.
\end{thm}
\begin{proof}
	The hyperbolic volume $V(z)$ of the ideal tetrahedron with modulus $z$ is given by the Bloch-Wigner function \cite{Za} 
	\begin{equation} \label{eq:volz}
		V(z) = \im \li(z) + \log|z| \arg(1-z).
	\end{equation}
	We only consider the case where a crossing is between different components.
	Let $V^{\pm}_c(a,b)$ be the sum of hyperbolic volumes of five ideal tetrahedra at a positive or negative crossing $c$ respectively.
	By using \eqref{eq:volz}, we can show that
	\[
		\im \Phi^+_c - V^+_c(a,b) = A^+_{j_1} \log|w_{j_1}|+A^+_{j_2} \log|w_{j_2}|+A^+_{j_3} \log|w_{j_3}|+A^+_{j_4} \log|w_{j_4}|,
	\]
	where $A^+_{j_i}$, with $i=1,2,3,4$ are:
	\begin{align*}
		A^+_{j_1}&= \frac{\pi}{2}(a-b) + \arg \left(1-e_a\frac{w_{j_1}}{w_{j_2}}\right)^{-1}\left(1-e_b^{-1}\frac{w_{j_1}}{w_{j_4}}\right)^{-1}\left( 1-\frac{w_{j_1}w_{j_3}}{w_{j_2}w_{j_4}}\right),\\
		A^+_{j_2}&= \frac{\pi}{2}(a+b) + \arg \left(1-e_a^{-1}\frac{w_{j_2}}{w_{j_1}}\right)\left(1-e_b^{-1}\frac{w_{j_2}}{w_{j_3}}\right)\left( 1-\frac{w_{j_2}w_{j_4}}{w_{j_1}w_{j_3}}\right)^{-1},\\
		A^+_{j_3}&= \frac{\pi}{2}(-a+b)+ \arg \left(1-e_a^{-1}\frac{w_{j_3}}{w_{j_4}}\right)^{-1}\left(1-e_b\frac{w_{j_3}}{w_{j_2}}\right)^{-1}\left( 1-\frac{w_{j_1}w_{j_3}}{w_{j_2}w_{j_4}}\right),\\
		A^+_{j_4}&= -\frac{\pi}{2}(a+b)+ \arg \left(1-e_a\frac{w_{j_4}}{w_{j_3}}\right)\left(1-e_b \frac{w_{j_4}}{w_{j_1}}\right)\left( 1-\frac{w_{j_2}w_{j_4}}{w_{j_1}w_{j_3}}\right)^{-1}.
	\end{align*}
	Similarly, we can show that
	\[
		\im \Phi^-_c - V^-_c(a,b) = A^-_{j_1} \log|w_{j_1}|+A^-_{j_2} \log|w_{j_2}|+A^-_{j_3} \log|w_{j_3}|+A^-_{j_4} \log|w_{j_4}|,
	\]
	where $A^-_{j_i}\ (i=1,2,3,4)$ are
	\begin{align*}
		A^-_{j_1}&= \frac{\pi}{2}(-a+b)+ \arg \left(1-e_a \frac{w_{j_1}}{w_{j_4}}\right)\left(1-e_b^{-1}\frac{w_{j_1}}{w_{j_2}}\right)\left( 1-\frac{w_{j_1}w_{j_3}}{w_{j_2}w_{j_4}}\right)^{-1},\\
		A^-_{j_2}&= \frac{\pi}{2}(a+b) + \arg \left(1-e_a \frac{w_{j_2}}{w_{j_3}}\right)^{-1}\left(1-e_b \frac{w_{j_2}}{w_{j_1}}\right)^{-1}\left( 1-\frac{w_{j_2}w_{j_4}}{w_{j_1}w_{j_3}}\right),\\
		A^-_{j_3}&= \frac{\pi}{2}(a-b) + \arg \left(1-e_a^{-1}\frac{w_{j_3}}{w_{j_2}}\right)\left(1-e_b \frac{w_{j_3}}{w_{j_4}}\right)\left( 1-\frac{w_{j_1}w_{j_3}}{w_{j_2}w_{j_4}}\right)^{-1},\\
		A^-_{j_4}&= -\frac{\pi}{2}(a+b)+ \arg \left(1-e_a^{-1}\frac{w_{j_4}}{w_{j_1}}\right)^{-1}\left(1-e_b^{-1}\frac{w_{j_4}}{w_{j_3}}\right)^{-1}\left( 1-\frac{w_{j_2}w_{j_4}}{w_{j_1}w_{j_3}}\right).
	\end{align*}
	By summing up over all crossings, we verify the proposition. 
\end{proof}
\begin{ex}[figure-eight knot]
Let $ \theta $ be a real number in $[0,\pi/3]$. The volume $V(\theta)$ of the cone-manifold of the figure-eight knot with a cone-angle $ \theta $ is 
given by the formula \cite{Med,MR}
\[
	V(\theta) = \int^{\frac{2 \pi }{3}}_{\theta} \arccosh (1+ \cos \theta - \cos 2 \theta)d \theta.
\]
In this case, the cone-manifold admits a hyperbolic structure.
On the other hand, the colored Jones polynomial for the figure-eight knot is given by Habiro and Le's formula \cite{Ha}
\[
	J_N(4_1;t) = \frac{1}{\{N\}}\sum^{N-1}_{p=0}\frac{\{N+p\}!}{\{N-p-1\}!}.
\]
We assume that $a$ is in $(5/6,1)$ so that $0 < 2 \pi(1-a)<\pi/3$.
The potential function of $J_{a(N)}(4_1,\xi _N)$ is
\[
	\Phi(a,x) = -2 \pi \iu a \log x -\li(e_a^2x)+\li(e_a^2x^{-1})
\]
and the derivative of this function with respect to $x$ is 
\[
	\frac{\partial \Phi}{\partial x} = \frac{1}{x} \log(-x+e_a^2+e_a^{-2}-x^{-1}).
\]
As a solution of the equation $ \partial \Phi/\partial x =0$, we obtain
\[
	x_0(a) = \left( \cos 2 \pi a - \frac{1}{2}\right)- \sqrt{\left( \cos 2 \pi a - \frac{3}{2}\right)\left( \cos 2 \pi a + \frac{1}{2}\right)}.
\]
Since $5/6 < a < 1$, the absolute value of $x_0(a)$ is equal to $1$. So we put $x_0(a) = e^{\iu \varphi(a)}$, where $\varphi(a) \in (-\pi,\pi]$.
Then, the imaginary part of $\Phi(a,x_0(a))$ is
\[
	\im \Phi(a,x_0(a)) = -2 \Lambda \left( \pi a + \frac{\varphi(a)}{2}\right)+2 \Lambda \left( \pi a - \frac{\varphi(a)}{2}\right).
\]
We will show that $ \im \Phi(a,x_0(a))= V(2 \pi(1-a))$ as a function on the closed interval $[2/3,1]$.
If $a = 2/3$, they are both $0$. The derivative with respect to $a$ is
	\begin{align*}
		\frac{d \Phi(a,x_0(a))}{da} &= \frac{\partial \Phi}{\partial a}(a,x_0(a)) + \frac{\partial \Phi}{\partial x}(a,x_0(a))\frac{d x_0(a)}{da}\\
		&=2 \pi \iu \log \frac{1-e_a^2 x_0(a)}{x_0(a)-e_a^2} \\
		&=-2 \pi^2 + 2 \pi \iu \log \left( \frac{e_a^2 x_0(a)-1}{e_a^{-2}x_0(a)-1}e_a^{-2}\right).
	\end{align*}
Since $e^{\iu \theta}-1 = 2 \sin (\theta/2)e^{\iu(\pi+ \theta)/2}$, we obtain
	\[
		\frac{d \Phi(a,x_0(a))}{da} = -2 \pi^2 + 2 \pi \iu \log \frac{\sin \frac{\varphi(a)+2 \pi a}{2}}{\sin \frac{\varphi(a)-2 \pi a}{2}}.
	\]
Let $f(a)$ be the function inside the $\log$, then
	\[
		\cosh \log f(a) = \frac{\sin^2 \frac{\varphi(a)+2 \pi a}{2} + \sin^2 \frac{\varphi(a)-2 \pi a}{2}}{2\sin \frac{\varphi(a)+2 \pi a}{2}\sin \frac{\varphi(a)-2 \pi a}{2}}.
	\]
Note that the denominator of the right-hand side is $ \cos(2 \pi a) -\cos \varphi(a) = 1/2$. Then,
	\begin{align*}
		\cosh \log f(a) &= 2 \left( \sin^2 \frac{\varphi(a)+2 \pi a}{2} + \sin^2 \frac{\varphi(a)-2 \pi a}{2}\right) \\
		&=2- \cos(\varphi(a) + 2 \pi a) - \cos(\varphi(a) - 2 \pi a) \\
		&=2-2 \cos \varphi(a)\cos 2 \pi a \\
		&=1 + \cos 2 \pi a -\cos 4 \pi a.
	\end{align*}
Therefore, we obtain
\[
	\frac{d \Phi(a,x_0(a))}{da} = -2 \pi^2 + 2 \pi \iu \arccosh(1 + \cos 2 \pi a -\cos 4 \pi a)
\]
Clearly, the imaginary part of this function is $ 2 \pi \arccosh(1 + \cos 2 \pi a -\cos 4 \pi a)$ which is equal to $dV(2\pi(1-a))/da$.
This shows that $ V(2\pi(1-a)) = \im \Phi(a,x_0(a))$.
\end{ex}
\begin{rem}
	We can show the following statement by the same procedure that appeared in \cite{Mh2}\footnote{In \cite{Mh2}, the value substituted for $t$ is slightly changed from the $N$-th root of unity.}:
	Let $a \in (5/12,1/2)$ be the limit of $a(N)/N$, where $(N \to \infty)$. Then, the limit
		\[
			4 \pi \lim _{N \to \infty} \frac{\log | J_{a(N)}(4_1;\xi_N^2)|}{N}
		\]
		is equal to the volume of the cone-manifold of the figure-eight knot with a cone-angle $2 \pi -4 \pi a$,
		where $N$ runs over all odd integers.
\end{rem}
\begin{ex}[Borromean rings]
	Let $K_B$ be the Borromean rings, $K_B(\alpha,\beta,\gamma)$ be the cone manifold of $K_B$ with cone-angles $\alpha,\ \beta,\ \gamma $,
	and $\Delta(\alpha,\theta)=\Lambda(\alpha+\theta)-\Lambda(\alpha-\theta)$, where $\Lambda(x)$ is the Lobachevsky function.
	If $0 < \alpha,\beta,\gamma < \pi$, then $K_B(\alpha,\beta,\gamma)$ admits a hyperbolic structure, and its volume is given by
	\[
		\vol K_B(\alpha,\beta,\gamma) = 2 \left(\Delta \left(\frac{\alpha}{2},\theta \right)+\Delta \left(\frac{\beta}{2},\theta \right)+\Delta \left(\frac{\gamma}{2},\theta \right)-2 \Delta \left(\frac{\pi}{2},\theta \right)-\Delta \left(0,\theta \right)\right),
	\]
	where $ \theta \in (0,\pi/2)$ is defined by the following conditions \cite{Med}:
	\begin{align*}
		T=&\tan \theta,\ L= \tan \frac{\alpha}{2},\ M= \tan \frac{\beta}{2},\ N= \tan \frac{\gamma}{2},\\
		&T^4-(L^2+M^2+N^2+1)T^2 -L^2M^2N^2=0.
	\end{align*}
	We define the function $ \tilde{\Delta}(x,y,z,\theta)$ by 
	\[
		\tilde{\Delta}(x,y,z,\theta) = 2 \left(\Delta \left(x,\theta \right)+\Delta \left(y,\theta \right)+\Delta \left(z,\theta \right)-2 \Delta \left(\frac{\pi}{2},\theta \right)-\Delta \left(0,\theta \right)\right)
	\]
	for convenience.
	On the other hand, the colored Jones polynomial for $K_B$ is given by \cite{Ha}
	\[
		J_{(l,m,n)}(K_B;t) = \sum^{\min(l,m,n)-1}_{i=1}\frac{\{l+i\}!\{m+i\}!\{n+i\}!(\{i\}!)^2}{\{1\}\{l-i-1\}!\{m-i-1\}!\{n-i-1\}!(\{2i+1\}!)^2}.
	\]
	Let $a,\ b$,and $c$ be the limit of $l/N,\ m/N$, and $n/N$ respectively.
	The potential function $\Phi_{K_B}(x)$ of $J_{(l,m,n)}(K_B;\xi_N)$ is 
	\begin{align*}
		\Phi_{K_B}(a,b,c,x) &= -2 \pi \iu(a+b+c)\log x +\frac{3}{2}(\log x)^2 \\
		&-\li(e_a^2 x)-\li(e_b^2 x)-\li(e_c^2 x)-2\li(x)\\
		&+\li \left(\frac{e_a^2}{x} \right)+\li \left(\frac{e_b^2}{x} \right)+\li \left(\frac{e_c^2}{x}\right)+2\li(x^2).
	\end{align*}
	The derivative of $\Phi_{K_B}(x)$ with respect to $x$ is
	\[
		x \frac{\partial \Phi_{K_B}}{\partial x}=\log \left(e_a^{-2}e_b^{-2}e_c^{-2}F(a,x)F(b,x)F(c,x)\frac{x^3(1-x)^2}{(1-x^2)^4}\right),
	\]
	where
	\[
		F(a,x) = (1-e_a^2 x)\left(1-\frac{e_a^2}{x} \right).
	\]
	Under the substistution $x = e^{2 \pi \iu \zeta}$, we obtain
	\begin{align*}
		&\frac{1}{2 \pi \iu}\frac{\partial \Phi_{K_B}}{\partial \zeta}\\
		&=\log \frac{\sin \pi(\zeta+a)\sin \pi(\zeta-a)\sin \pi(\zeta+b)\sin \pi(\zeta-b)\sin \pi(\zeta+c)\sin \pi(\zeta-c)}{\sin^2 \pi \zeta \cos^4 \pi \zeta}\\
		&=\log \frac{\tan^2 \pi \zeta-A^2}{1+A^2}\frac{\tan^2 \pi \zeta-B^2}{1+B^2}\frac{\tan^2 \pi \zeta-C^2}{1+C^2}\frac{1}{\tan^2 \pi \zeta},
	\end{align*}
	where $A = \tan \pi(1-a),\ B = \tan \pi(1-b)$, and $C = \tan \pi(1-c)$. Therefore, if $\pm \tan \pi \zeta $ are solutions of the equation
	\[
		\frac{t^2-A^2}{1+A^2}\frac{t^2-B^2}{1+B^2}\frac{t^2-C^2}{1+C^2}\frac{1}{t^2} = 1,
	\]
	which is equivalent to the equation
	\[
		(t^2+1)(t^4 -(A^2+B^2+C^2+1)t^2 - A^2 B^2 C^2) = 0,
	\]
	then $x = e^{2 \pi \iu \zeta}$ is a saddle point of $\Phi_{K_B}(a,b,c,x)$. 
	By using the properties of the Lobachevsky function, such as
	\begin{align*}
		\li(e^{2 \iu \theta}) &= \frac{\pi^2}{6}-\theta(\pi-\theta)+2 \iu \Lambda(\theta), \\
		\Lambda(2\theta) &= 2 \Lambda(\theta)+2 \Lambda \left( \theta + \frac{\pi}{2}\right),
	\end{align*}
	we obtain
	\begin{align*}
		\im \Phi _{K_B}(a,b,c,e^{2 \pi \iu \zeta}) &= \tilde{\Delta}\left(\pi(1-a),\pi(1-b),\pi(1-c),\pi(1-\zeta)\right)\\
		&= \vol K_B(2 \pi(1-a),2 \pi(1-b),2 \pi(1-c)).
	\end{align*}
\end{ex}
\section{The completeness condition} \label{sec:compness}
In the previous section, we fixed $a_1,\ldots,a_n$. In this section, we regard them as variables and find a geometric meaning.
First, we consider the case where a crossing is between different components.
The derivatives of the potential function with respect to the parameters corresponding to the colors are
	\begin{align*}
		\frac{\partial \Phi^+_c}{\partial a} &=\frac{\pi \iu}{2} \log \left( 1- e_a \frac{w_{j_4}}{w_{j_3}}\right)\left( 1- e_a \frac{w_{j_1}}{w_{j_2}}\right)^{-1}\left( 1- e_a^{-1} \frac{w_{j_3}}{w_{j_4}}\right)\left( 1- e_a^{-1} \frac{w_{j_2}}{w_{j_1}}\right)^{-1},\\
		\frac{\partial \Phi^+_c}{\partial b} &=\frac{\pi \iu}{2} \log \left( 1- e_b \frac{w_{j_4}}{w_{j_1}}\right)\left( 1- e_b \frac{w_{j_3}}{w_{j_2}}\right)^{-1}\left( 1- e_b^{-1} \frac{w_{j_1}}{w_{j_4}}\right)\left( 1- e_b^{-1} \frac{w_{j_2}}{w_{j_3}}\right)^{-1},\\
		\frac{\partial \Phi^-_c}{\partial a} &=\frac{\pi \iu}{2} \log \left( 1- e_a \frac{w_{j_1}}{w_{j_4}}\right)\left( 1- e_a \frac{w_{j_2}}{w_{j_3}}\right)^{-1}\left( 1- e_a^{-1} \frac{w_{j_4}}{w_{j_1}}\right)\left( 1- e_a^{-1} \frac{w_{j_3}}{w_{j_2}}\right)^{-1},\\
		\frac{\partial \Phi^-_c}{\partial b} &=\frac{\pi \iu}{2} \log \left( 1- e_b \frac{w_{j_3}}{w_{j_4}}\right)\left( 1- e_b \frac{w_{j_2}}{w_{j_1}}\right)^{-1}\left( 1- e_b^{-1} \frac{w_{j_4}}{w_{j_3}}\right)\left( 1- e_b^{-1} \frac{w_{j_1}}{w_{j_2}}\right)^{-1}.
	\end{align*}
We can observe the correspondence between these derivatives and dilation components by cusp diagrams (Figure \ref{fig:cdl}).
	\begin{figure}[htb]
		\centering
		\begin{minipage}{0.45\hsize}
			\centering
			\includegraphics[height=4.5cm]{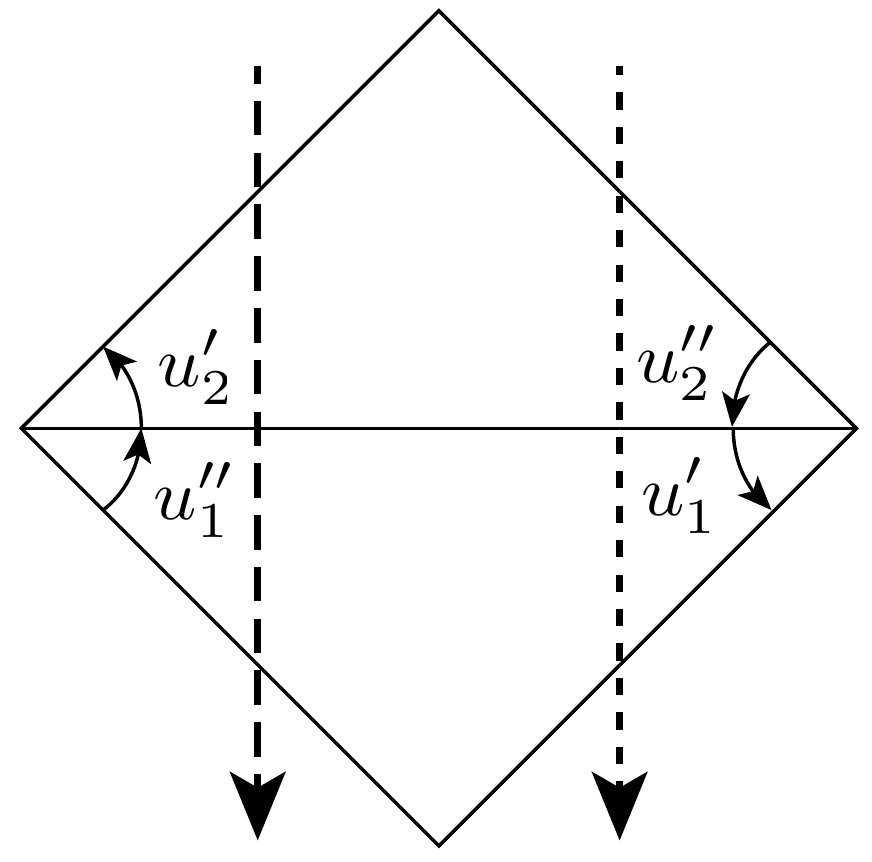}
			\subcaption{Upper side of a positive crossing.}
			\label{fig:cdpol}
		\end{minipage}
		\begin{minipage}{0.45\hsize}
			\centering
			\includegraphics[height=4.5cm]{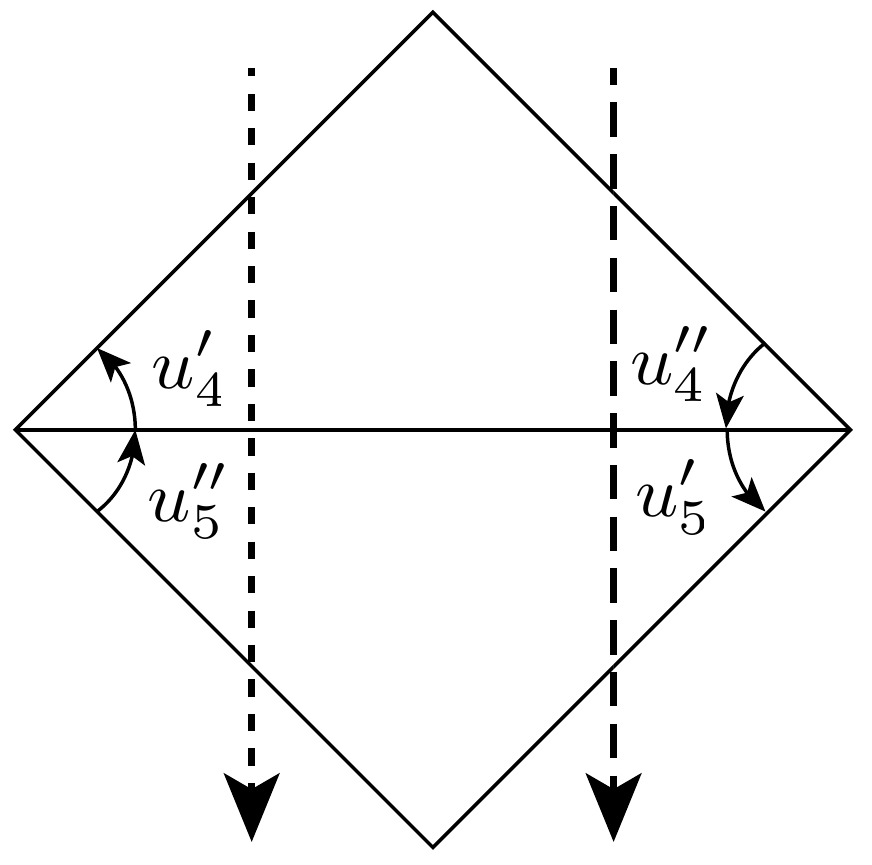}
			\subcaption{Lower side of a positive crossing.}
			\label{fig:cdpul}
		\end{minipage}\\
		\begin{minipage}{0.45\hsize}
			\centering
			\includegraphics[height=4.5cm]{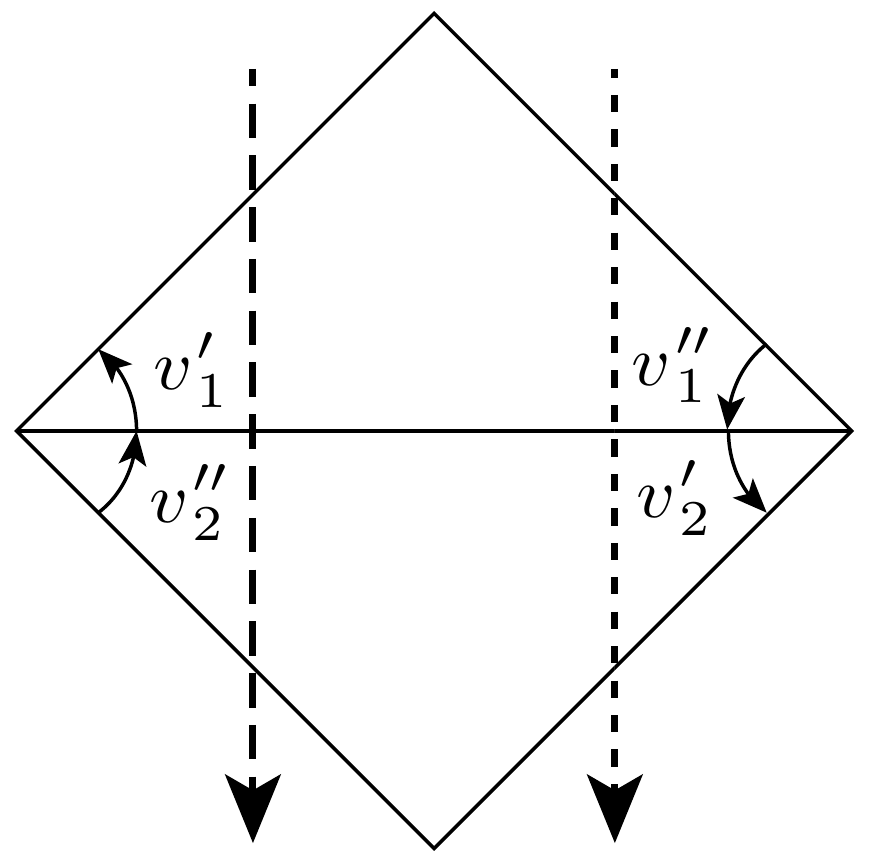}
			\subcaption{Upper side of a negative crossing.}
			\label{fig:cdnol}
		\end{minipage}
		\begin{minipage}{0.45\hsize}
			\centering
			\includegraphics[height=4.5cm]{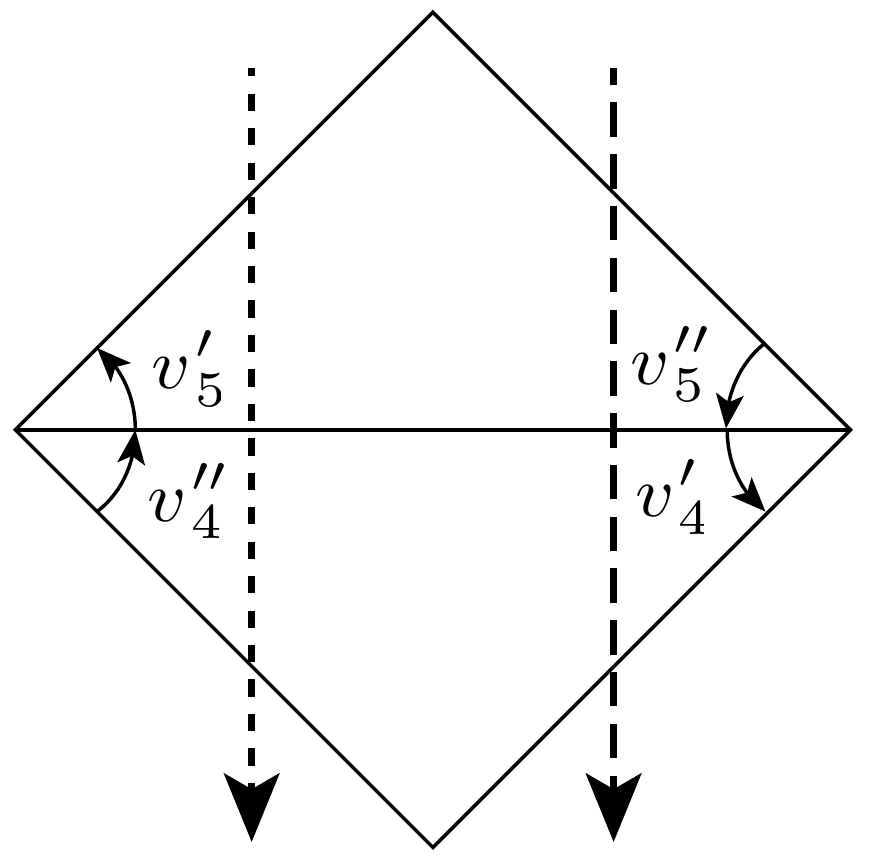}
			\subcaption{Lower side of a negative crossing.}
			\label{fig:cdnul}
		\end{minipage}
		\caption{Cusp diagrams.}
		\label{fig:cdl}
	\end{figure}
In Figure \ref{fig:cdl}, $\partial\Phi_c^+/\partial a$ corresponds to the upper side of a positive crossing (Figure \ref{fig:cdpol}),
$\partial\Phi_c^+/\partial b$ to the lower side of a positive crossing (Figure \ref{fig:cdpul}),
$\partial\Phi_c^-/\partial a$ to the upper side of a negative crossing (Figure \ref{fig:cdnol}),
and $\partial\Phi_c^-/\partial b$ to the upper side of a negative crossing (Figure \ref{fig:cdnul}).
A similar correspondence holds in the case where a crossing is between the same component.
Let $l_i$ be the longitude that is parallel to the component $L_i$, and let $\tilde{l}_i$ be the longitude of the component $L_i$ with $\lk(\tilde{l}_i,L_i)=0$.
For a curve $\gamma$ on the cusp diagram, we define $\delta(\gamma)$ as the dilation component of $\gamma$.
Then, by the above observation
\[
	\exp \left(\frac{1}{\pi \iu} \frac{\partial \Phi'_D}{\partial a_i}\right) = \exp \left(\frac{1}{2} \log \delta(l_i)^2 \right) = \delta(l_i),
\]
where $\Phi'_D$ is a potential function of the colored Jones polynomial without the modification for the Reidemeister move I.
Next, we consider the contribution of the modification.
For a positive crossing between the same component with a parameter $a$,
the modification corresponds to the addition of $(\pi \iu a)^2$, and its derivative is
\[
	\frac{1}{\pi \iu}\frac{d}{da}(\pi \iu a)^2 = 2 \pi \iu a = \log e_a^2.
\]
Here, $e_a^2$ is equal to the dilation component of the meridian with the inverse orientation.
Similarly, for a negative crossing, the derivative of $-(\pi \iu a)^2$ corresponds to the dilation component of the meridian.
Therefore,
\begin{equation}\label{eq:delaihol}
	\exp \left(\frac{1}{\pi \iu} \frac{\partial \Phi _D}{\partial a_i}\right) = \delta (\tilde{l}_i).
\end{equation}
\begin{rem}
	If $K$ is a knot, we have a more simple correspondence. The derivatives of $\Phi^{\pm}_c$ with respect to $a$ are 
	\begin{align}
	\begin{split}\label{eq:delpak}
		\frac{1}{\pi \iu} &\frac{\partial \Phi^+_c}{\partial a}=\log e_a^2 \\
		&+ \log \left( 1- e_a^{-1} \frac{w_{j_3}}{w_{j_4}}\right)\left( 1- e_a \frac{w_{j_4}}{w_{j_1}}\right)\left( 1- e_a^{-1} \frac{w_{j_2}}{w_{j_1}}\right)^{-1}\left( 1- e_a \frac{w_{j_3}}{w_{j_2}}\right)^{-1},\\
	\end{split}
	\end{align}
	\begin{align}
	\begin{split}\label{eq:delnak}
		\frac{1}{\pi \iu} &\frac{\partial \Phi^-_c}{\partial a}=\log e_a^{-2}\\
		&+\log \left( 1- e_a^{-1} \frac{w_{j_4}}{w_{j_1}}\right)\left( 1- e_a \frac{w_{j_3}}{w_{j_4}}\right)\left( 1- e_a^{-1} \frac{w_{j_3}}{w_{j_2}}\right)^{-1}\left( 1- e_a \frac{w_{j_2}}{w_{j_1}}\right)^{-1}.
	\end{split}
	\end{align}
	The second term of equation \eqref{eq:delpak} corresponds to the upper side and the lower side of a positive crossing (Figures \ref{fig:cdpok} and \ref{fig:cdpuk}),
	and the second term of equation \eqref{eq:delnak} corresponds to the upper side and the lower side of a negative crossing (Figures \ref{fig:cdnok} and \ref{fig:cdnuk})
	\begin{figure}[htb]
		\centering
		\begin{minipage}{0.45\hsize}
			\centering
			\includegraphics[height=4.5cm]{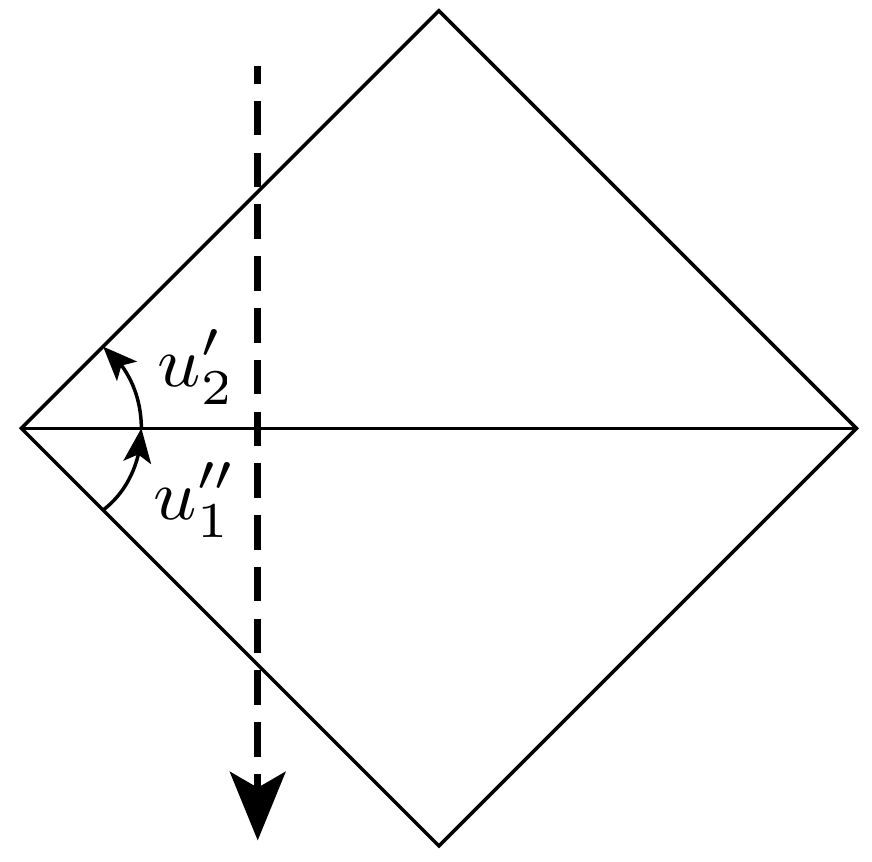}
			\subcaption{Upper side of a positive crossing.}
			\label{fig:cdpok}
		\end{minipage}
		\begin{minipage}{0.45\hsize}
			\centering
			\includegraphics[height=4.5cm]{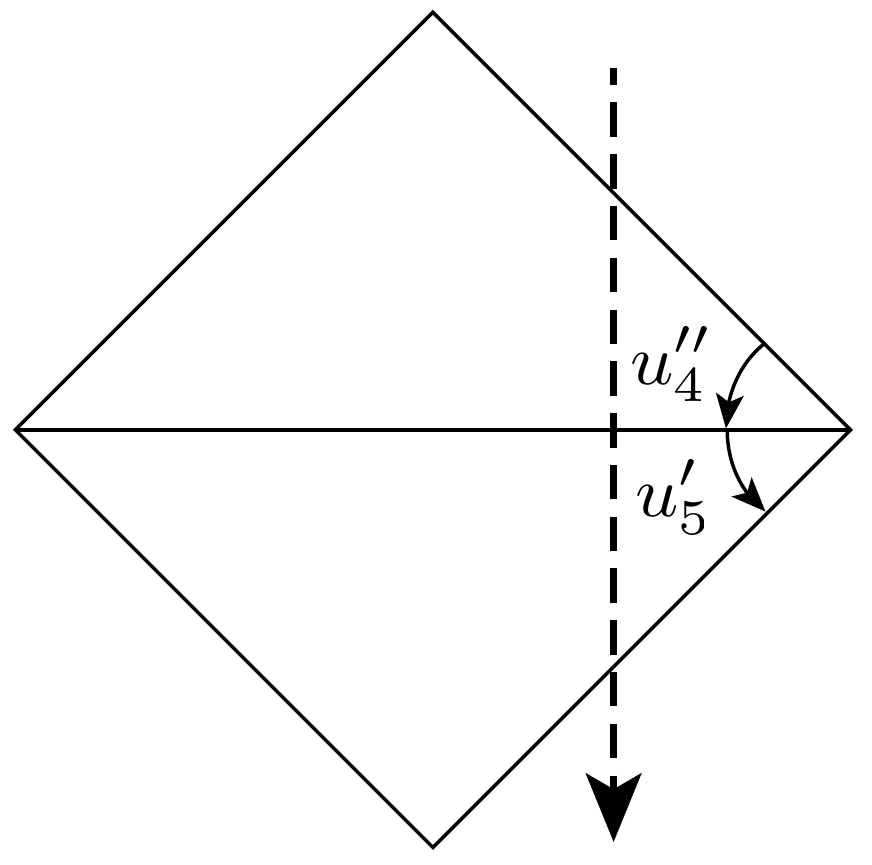}
			\subcaption{Lower side of a positive crossing.}
			\label{fig:cdpuk}
		\end{minipage}\\
		\begin{minipage}{0.45\hsize}
			\centering
			\includegraphics[height=4.5cm]{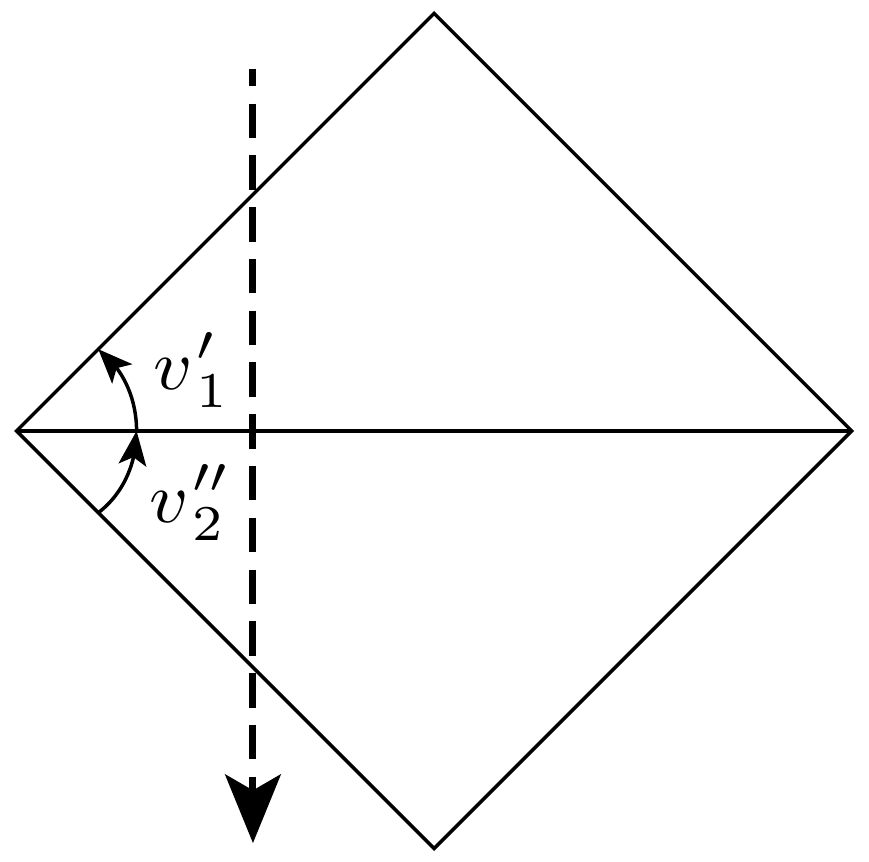}
			\subcaption{Upper side of a negative crossing.}
			\label{fig:cdnok}
		\end{minipage}
		\begin{minipage}{0.45\hsize}
			\centering
			\includegraphics[height=4.5cm]{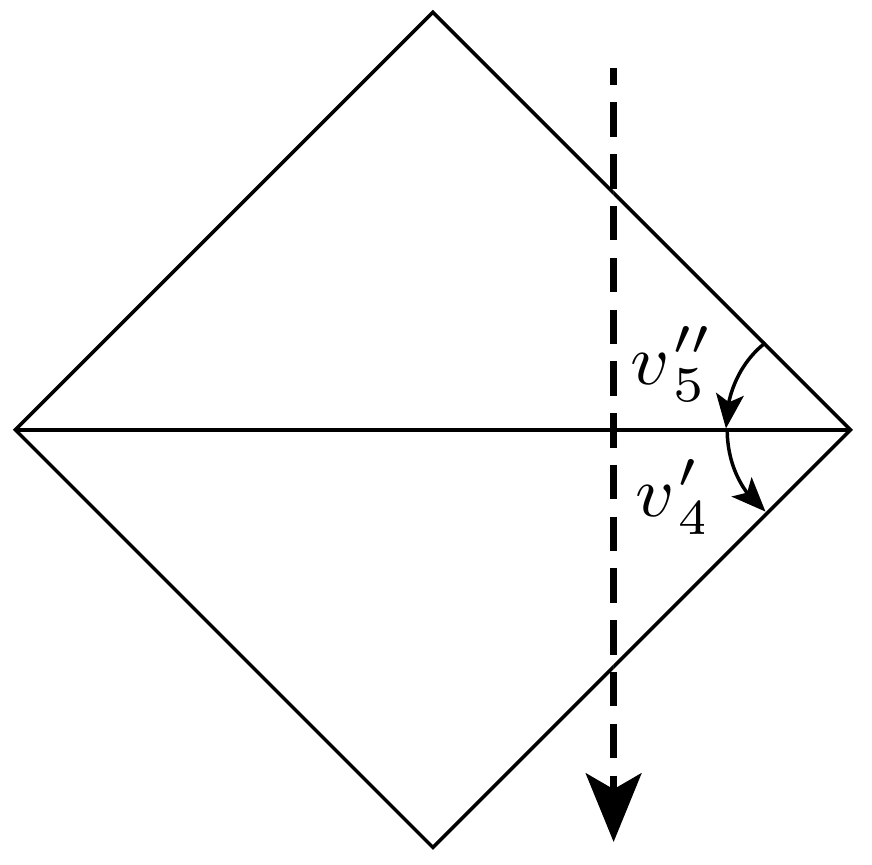}
			\subcaption{Lower side of a negative crossing.}
			\label{fig:cdnuk}
		\end{minipage}
		\caption{Cusp diagrams of a knot complement.}
	\end{figure}
\end{rem}
\begin{rem}
	Changing the variable $a_i$ to $u_i = 2\pi \iu a_i$, we have
	\[
		2 \frac{\partial \Phi _D}{\partial u_i} = \frac{1}{\pi \iu}\frac{\partial \Phi _D}{\partial a_i}.
	\]
	Then, 
	\[
		\Psi(\boldsymbol{u}) = 4 \left(\Phi _D(\boldsymbol{u},\sigma _1(\boldsymbol{u}),\ldots,\sigma _\nu(\boldsymbol{u})) - \Phi _D(\boldsymbol{0},\sigma _1(\boldsymbol{0}),\ldots,\sigma _\nu(\boldsymbol{0}))\right),
	\]
	where $ \boldsymbol{u}=(u_1,\ldots,u_n)$ and $\boldsymbol{0}=(0,\ldots,0)$ satisfy the conditions of the Neumann-Zagier potential function \cite{NZ}.
	Namely, $\Psi$ satisfies $\Psi(\boldsymbol{0})=0$ and
	\[
		\frac{1}{2} \frac{\partial \Psi}{\partial u_i} = \log \delta (\tilde{l}_i).
	\]
\end{rem}
If $a_i =1$, with $i=1,\ldots,n$, all the dilation components of meridians are $1$.
Furthermore, the contributions of parts, as shown in Figure \ref{fig:cdmerid} to the dilation component of the longitude is $1$ hence $\delta (\tilde{l}_i) =1,\ (1,\ldots,n)$. 
Therefore, the point $(\boldsymbol{1},\sigma_1(\boldsymbol{1}),\ldots,\sigma_\nu (\boldsymbol{1}))$
gives a complete hyperbolic structure to the link complement \cite{BP}, where $\boldsymbol{1}=(1,\ldots,1)$.
Moreover, by the equations \eqref{eq:lsaddle} and \eqref{eq:delaihol} the point is a solution of the following system of equations:
\[
	\begin{dcases}
		\exp \left(w_i \frac{\partial \Phi _D}{\partial w_i} \right)= 1, &(i = 1,\ldots,\nu)\\
		\exp \left(\frac{1}{\pi \iu}\frac{\partial \Phi _D}{\partial a_j}\right) = 1, &(j = 1,\ldots,n)
	\end{dcases}
\]
Hence, we obtain the following theorem:
\begin{thm} \label{thm:m2}
	Let $D$ be a diagram of a hyperbolic link with $n$ components, and let $\boldsymbol{1}$ be $(1,\ldots,1) \in \mathbb{Z}^n$.
	The point $(\boldsymbol{1},\sigma_1(\boldsymbol{1}),\ldots,\sigma_\nu (\boldsymbol{1}))$ is a saddle point of the function
	$\Phi_D (a_1,\ldots,a_n,w_1,\ldots,w_\nu)$ and gives a complete hyperbolic structure to the link complement.
\end{thm}

\section{The Witten-Reshetikhin-Turaev invariant} \label{sec:wrt}
In \cite{KM}, the Witten-Reshetikhin-Turaev invariant for the manifold obtained by Dehn surgery on a link is stated.
Furthermore, Murakami \cite{Mh} calculated the optimistic limit of the Witten-Reshetikhin-Turaev invariant
for the manifold obtained by integer surgery on the figure-eight knot.
By a similar argument as in Section \ref{sec:nchypstr}, we would be able to explain
the correspondence of the Witten-Reshetikhin-Turaev invariant and the geometry of the manifold obtained by Dehn surgery on a link.
The procedure might be as follows:
The Witten-Reshetikhin-Turaev invariant for the manifold $M_{f_1,\ldots,f_n}$ obtained by Dehn surgery on a link $L=L_1 \cup \cdots \cup L_n$
with a framing $f_i$ on $L_i$, where $i=1,\ldots,n$ can be written as a summation of the colored Jones polynomial $J_{\boldsymbol{k}}(L;\xi_N)$
multiplied by $t^{-\frac{1}{4}\sum f_j k_j^2}$, where $\boldsymbol{k}=(k_1,\ldots,k_n)$ are colors of $L$.
See \cite{KM} for details, but note that $t$ in \cite{KM} and $t$ in this paper are different.
We suppose that $M_{f_1,\ldots,f_n}$ admits a hyperbolic structure.
Let $ \alpha_i$ be $e^{\pi \iu a_i}$ and regard it as a complex parameter that is not necessarily in the unit circle. Then, we have
\[
	\frac{1}{\pi \iu} \frac{\partial \Phi}{\partial a_i} = \alpha _i \frac{\partial \Phi}{\partial \alpha _i}.
\]
Multiplying $t^{-\frac{1}{4}\sum f_j k_j^2}$ leads to the addition of $-\sum f_j (\log \alpha _j)^2$ to the potential function.
The derivative of it with respect to $\alpha _i$ is
\[
	\alpha_i \frac{\partial}{\partial \alpha_i} \left( -\sum^n_{j=1} f_j (\log \alpha_j)^2 \right)=-2 f_i \log \alpha_i = \log \alpha_i^{-2f_i}.
\]
Then, the saddle point equation is equivalent to the system of equations consisting of the gluing condition and
\[
	\delta(\tilde{l}_i) = \alpha _i^{2f_i},\quad (i=1,\ldots,n).
\]
Recall that the dilation component of the meridian $m_i$ of $L_i$ is $\alpha_i^{-2}$, which implies that $\delta(m_i)^{-f_i} = \delta(\tilde{l}_i)$.
If we suppose that $|\alpha_i|$ is less than $1$ and $f_i$ is a positive integer, the developing image would be as shown in Figure \ref{fig:wrtdev}.
By filling in the singular set, the developing image becomes complete.
	\begin{figure}[htb]
		\centering
		\includegraphics[scale=0.7]{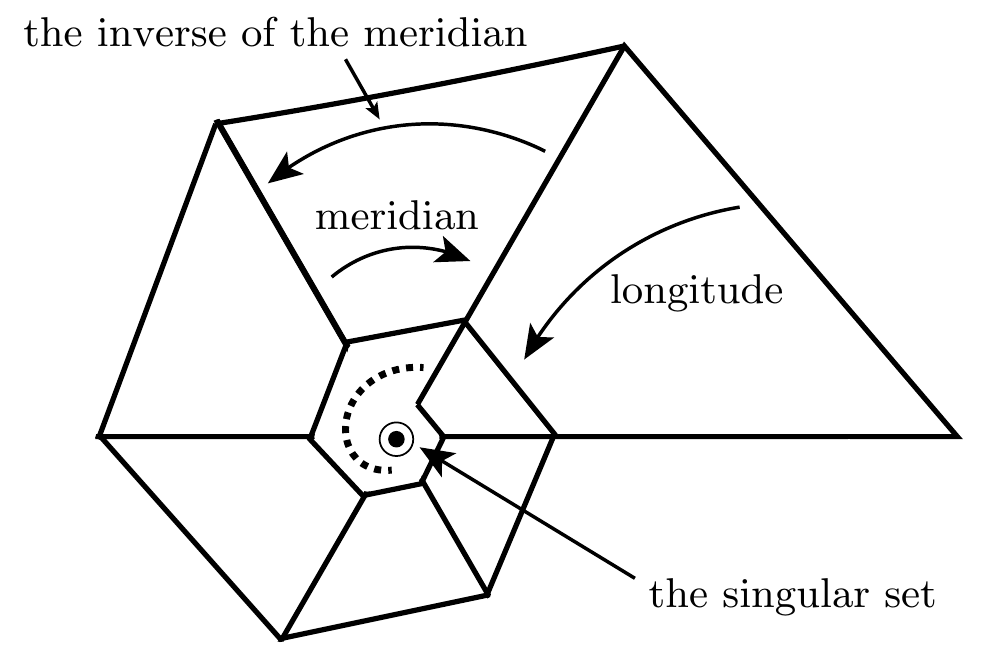}
		\caption{The schematic diagram of the developing image in the case of $f_i = 6$.}
		\label{fig:wrtdev}
	\end{figure}

\end{document}